\documentclass[10pt,a4paper,reqno, oneside]{amsart}
\usepackage[foot]{amsaddr}
\usepackage{graphicx}

\makeatletter
\renewcommand\subsection{\@startsection{subsection}{1}
	\z@{.5\linespacing\@plus.7\linespacing}{-.5em}
	{\normalfont\scshape}}
\makeatother

\RequirePackage[utf8]{inputenc}
\RequirePackage[T1]{fontenc}
\RequirePackage[english]{babel}

\RequirePackage[toc, page]{appendix}

\RequirePackage[left=3cm,right=5cm,top=4cm,bottom=4cm]{geometry}
\RequirePackage{amsmath, bbm, amsfonts, amssymb, amsthm} % packages for mathematics
\RequirePackage{hyperref, verbatim}
\RequirePackage{enumerate}  % enumerate with custom symbols
\RequirePackage{xcolor, extarrows}
\RequirePackage{graphicx, tikz, tikz-cd}
\RequirePackage[autostyle=true]{csquotes} % for quotation
\RequirePackage{aligned-overset}
\RequirePackage{breakcites}
\RequirePackage{breqn}
\RequirePackage{url}
\RequirePackage{enumitem}
\usepackage{nomencl}
\newcommand{\ml}{\left\{ } % left set bracket
\newcommand{\mr}{\right\} } % right set bracket
\newcommand{\setv}[2]{\left\{ #1 \,\middle|\, #2 \right\}} % set with bracket and vertical line
\newcommand{\kl}{\left(} % left round bracket
\newcommand{\kr}{\right)} % right round bracket
\newcommand{\abs}[1]{\left| #1 \right|} % absolute value in one command
\newcommand{\il}{\left [} % left square bracket
\newcommand{\ir}{\right ]} % right square bracket
\newcommand{\floor}[1]{\left\lfloor #1 \right\rfloor} % floor
\newcommand{\N}{\mathbb{N}} % natural numbers
\newcommand{\Z}{\mathbb{Z}} % whole numbers
\newcommand{\R}{\mathbb{R}} % real numbers
\newcommand{\C}{\mathbb{C}} % complex numbers
\newcommand{\symdiff}{\bigtriangleup} % symmetric difference

\newcommand{\Per}{\mathrm{Per}}	% Periodic positions of Toeplitz sequence
\newcommand{\tm}{\subseteq}
\newcommand{\mt}{\supseteq}

\newcommand{\htop}{h_{\mathrm{top}}}

\newtheorem{thm}{Theorem}[section]
\newtheorem{lem}[thm]{Lemma}
\newtheorem{prop}[thm]{Proposition}

\newtheorem{cor}[thm]{Corollary}

\newtheorem*{thm*}{Theorem}

\newtheoremstyle{named}{}{}{\itshape}{}{\bfseries}{.}{.5em}{\thmnote{#3}}
\theoremstyle{named}
\newtheorem*{namedtheorem}{Theorem}

\theoremstyle{definition}

\newtheorem{rem}[thm]{Remark}

\newenvironment{construction}[1][Construction]
	{\begin{proof}[#1]
		}
	{\end{proof}}

\newcommand{\1}{\mathbbmss{1}} % 1 with double stroke

\newcommand{\ol}[1]{\overline{#1}} % overline

\numberwithin{equation}{section}

\title{Amorphic Complexity and Entropy for Symbolic Model Sets}
\author{Jamal Drewlo}
\address{Institute of Mathematics, Friedrich Schiller University Jena, 07743 Jena, Germany}
\email{jamal.drewlo@uni-jena.de}

\begin{document}

\begin{abstract}
We show a continuity result for the Weyl pseudometric on subshifts which are generated by model sets.
This fact is then used for multiple constructions of subshifts that exhibit different behavior regarding entropy, amorphic complexity and their maximal equicontinuous factor.\\[2pt]
\end{abstract}

\maketitle

\tableofcontents

\section{Introduction}
The \textit{cut and project method} is one of the most versatile and powerful construction schemes in the study of aperiodic order and mathematical quasicrystals.
Since their introduction by Meyer in \cite{Me1972Algebraic}, cut and project schemes have been a focal point in numerous studies -- especially on their dynamical properties (see \cite{Sc2000GeneralizedModel, LaPl2003Repetitive, BaLeMo2007Model, BjHaPo2018AperiodicOrder, JaLeOe2019ModelPositive, FuGlJaOe2021IrregularModel} to name a few).
In this work, we restrict our viewpoint to specific cut and project schemes of the form $(G,H,\mathcal{L})$, where $G$ is a countably infinite, amenable group, $H$ is a metrizable group compactification of $G$ with embedding $\tau: G \to H$.
Furthermore, the \textit{lattice} $\mathcal{L}\tm G\times H$ is given by the graph of $\tau$, i.e.\ $\mathcal{L}= \{(g,\tau(g)): g\in G\}$.
In this setting, one can naturally associate to any subset $W\tm H$ a $G$-array $x_W := \1_W\circ \tau \in \{0,1\}^G$, which shall be called a \textit{symbolic model set}.
Here, $\1_G$ denotes the indicator function of $W$.\\
Our first goal is to examine symbolic model sets and their associated subshifts with respect to the so-called \textit{Besicovitch} and \textit{Weyl pseudometric}.
Both these pseudometrics represent a notion of average distance and -- especially in the context of symbolic dynamics -- present classical objects of study (see for example \cite{DoIw1988QuasiUniform, CaFoMaMa1997Besicovitch, BlFoKu1997Cellular, LackaStraszak2018QuasiUniform}).
Regarding the first goal, we make the following fairly straightforward observation.
\begin{thm}\label{thm: properties Phi intro}
    The set $\mathfrak{B}(H)$ of Borel-measurable subsets of $H$ can be naturally equipped with pseudometrics $D$ and $\ol{D}$, given by
    \[ D(A,B) = \nu(A\symdiff B) \text{ and } \ol{D}(A,B) = \nu(\ol{A\symdiff B}), \]
    where $\nu$ is the normalized Haar measure on $H$.
    Moreover, let $\mathcal{F}$ be a F\o lner sequence in $G$ and equip $\{0,1\}^G$ with the Besicovitch pseudometric $D_{\mathcal{F}}$.
    Then
\begin{enumerate}
	\item The map $\Phi: \mathfrak{B}(H)\to \{0,1\}^G, \:W \mapsto x_W$ is continuous w.r.t.\ $\ol{D}$ and $D_{\mathcal{F}}$.
	\item Restricting $\Phi$ to the regular windows in $H$ yields an isometry between $D$ and $D_{\mathcal{F}}$.
\end{enumerate}
Moreover, (1) and (2) also hold with $D_{\mathcal{F}}$ replaced by the Weyl pseudometric $D_{\mathrm{Weyl}}$.
\end{thm}
It turns out that there is an intimate connection between the Weyl (resp.\ Besicovitch) pseudometric and topological entropy (resp.\ amorphic complexity) \cite{LackaStraszak2018QuasiUniform, FuGrJa2016Amorphic}.
Both topological entropy and amorphic complexity are means of measuring the chaoticity of a dynamical system.
However, while entropy takes finite values even on systems which can be considered fairly chaotic (e.g.\ the full shift) amorphic complexity is finite-valued only for systems of low complexity (in particular of zero entropy).
A natural question to ask is what values those quantities can attain for a specific subclass of systems.
For example, in \cite{Ku2024Amorphic} it is shown that any non-negative, possibly infinite value can be obtained by a (not necessarily minimal) $\Z$-subshift on two symbols.
In this paper, we aim to give answers to the following two questions:
\begin{itemize}
    \item What values of amorphic complexity can be obtained with a regular Toeplitz $G$-subshift?
    \item What values of topological entropy can be obtained with an almost automorphic $G$-subshift?
\end{itemize}
In order to answer both questions, we will use the fact that the class of symbolic model sets is versatile enough to construct a vast range of almost automorphic (Toeplitz) subshifts.
Regarding the first question, we are able to show for a large class of groups $G$ that any finite value greater or equal to $1$ can be attained.
\begin{thm}\label{thm: realize ac Toeplitz intro}
    Let $G$ be a finitely generated, torsion-free, nilpotent group and let $a\in [1,\infty)$.
    Then there exists a regular Toeplitz array $x\in \{0,1\}^G$ whose Toeplitz subshift $(\ol{O_G(x)},G)$ admits amorphic complexity of value $a$.
\end{thm}
We will see in the proof in Section~\ref{sec: krieger/ac} that one also has some degree of freedom for choosing the maximal equicontinuous factor of $(\ol{O_G(x)},G)$.
It should be mentioned that in \cite[Example 3.16, 3.17]{KasjanKeller2025Besicovitch}, using techniques for $\mathcal{B}$-free subshifts, the authors prove the same conclusion as in Theorem~\ref{thm: realize ac Toeplitz intro} but only for $G=\Z$ (however allowing also $a = \infty$). \\
We are furthermore able to prove an upper bound on amorphic complexity, which is a direct generalization of \cite[Thm.\ 1.6]{FuGrJa2016Amorphic}.
\begin{thm}\label{thm: estimate ac Toeplitz intro}
    Let $G$ be a countably infinite, residually finite amenable group, let $\mathbf{\Gamma}=(\Gamma_n)_{n\in\N}$ be a decreasing sequence of finite index normal subgroups of $G$ and let $x\in \{0,1\}^G$ be a regular Toeplitz array which is relative to $\mathbf{\Gamma}$.
    Then, one has
    \[ \ol{\mathrm{ac}}(\ol{O_G(x)},G) \leq \limsup_{n\to\infty} \frac{\log [G:\Gamma_{n+1}]}{-\log(1-D(x,\Gamma_n))},\]
    where $\ol{\mathrm{ac}}(\ol{O_G(x)},G)$ denotes the upper amorphic complexity of the subshift $(\ol{O_G(x)},G)$ and $D(x,\Gamma_n)$ denotes the asymptotic density of $\Per(x,\Gamma_n)$ in $G$.
\end{thm}
The second question as it is formulated above, can already be answered using Krieger's Theorem \cite[Thm.\ 1.1]{Krieger2007Toeplitz}, which shows that any possible entropy can be attained by a Toeplitz subshift -- and those are always almost automorphic.
However, Krieger's result lacks specification on the maximal equicontinuous factor of the subshift, apart from the fact that it is a $G$-odometer.
We are able to strengthen this by showing that not only any $G$-odometer, but in fact any metric compactification of $G$ can be realized as the maximal equicontinuous factor.
\begin{thm}\label{thm: realize entropy intro}
    Let $G$ be a countably infinite, residually finite amenable group, let $(H,\tau)$ be a metric compactification of $G$ and let $h\in [0,\log 2)$.
    Then, there exists an array $x\in \{0,1\}^G$ which is almost automorphic over $H$ such that $\htop(\ol{O_G(x)},G) = h$.
\end{thm}
\textit{Structure of the article.}
Several basic notions, such as amorphic complexity, are introduced in Section~\ref{sec: prelims}.
In Section~\ref{sec: odometers} we discuss $G$-odometers for countably infinite residually finite groups and outline an alternative representation of $G$-odometers.
This will be of importance for later constructions.
Section~\ref{sec: Toeplitz} introduces the notion of a Toeplitz array and necessary background.
The definition of a cut and project scheme as well as of a model set is given in Section~\ref{sec: model}, followed by the proof of Theorem~\ref{thm: properties Phi intro}.
Section \ref{sec: ac} is dedicated to proving Theorem~\ref{thm: realize ac Toeplitz intro} and Theorem~\ref{thm: estimate ac Toeplitz intro}.
Moreover, it turns out that with similar techniques we can construct a regular window whose boundary is a singleton but whose associated subshift admits infinite amorphic complexity (see Proposition~\ref{prop: counterexample}).
This is a counterexample to \cite[Thm.\ 1.3]{FuGrJaKw2023Amorphic} in the case of an odometer as the internal group.
Lastly, in Section \ref{sec: krieger} we begin by constructing paths which are continuous w.r.t.\ the pseudometric $\ol{D}$.
This is then used to show Theorem~\ref{thm: realize entropy intro}.

\section{Preliminaries}\label{sec: prelims}
\subsection{Basic definitions and background}\label{sec: definitions and background}
Let $G$ be a countably infinite, amenable group (see \cite{Pier1984Amenable} for further reading on amenable groups).
A \textbf{dynamical system} $(X,\alpha,G)$ consists of a compact metric space $(X,d)$ and an action $\alpha: G\times X\to X,\: (g,x)\mapsto \alpha_g(x)$ such that for each $g\in G$ the mapping $x\mapsto \alpha_g(x)$ is continuous.
We often suppress the action $\alpha$ and write $gx$ instead of $\alpha_g(x)$ and $(X,G)$ instead of $(X,\alpha,G)$ for the dynamical system.
We denote the \textbf{$G$-orbit} of $x\in G$ by $O_G(x) = \{gx: g\in G\}$.
One calls $(X,G)$ \textbf{minimal} if every $G$-orbit is dense in $X$, i.e.\ $\ol{O_G(x)} = X$ for all $x\in X$.
If $(Y,G)$ is another dynamical system, we call a map $\varphi: X\to Y$ a \textbf{factor map} if $\varphi$ is surjective, continuous and $\varphi(gx) = g\varphi(x)$ holds for all $g\in G, x\in X$.
In this case, $(Y,G)$ is called a \textbf{factor} of $(X,G)$.
If $\varphi$ is additionally bijective, one calls $\varphi$ a \textbf{conjugacy}.\\
An \textbf{invariant measure} of $(X,G)$ is a Borel probability measure $\mu$ on $X$ which satisfies $\mu(gA)= \mu(A)$ for all $g\in G$ and measurable sets $A$.
If $(X,G)$ has exactly one invariant measure, we call the system \textbf{uniquely ergodic}.
The following is well-known.
\begin{thm}[Uniform Ergodic Theorem]\label{eq: uniform ergodic theorem}
    Let $\mathcal{F}=(F_n)_{n\in\N}$ be a F\o lner sequence in $G$.
    Let $(X,G)$ be a uniquely ergodic dynamical system with unique invariant measure $\mu$.
    Then, for any $\varphi\in C(X)=\{\psi:X\to \C: \psi\text{ is continuous}\}$, the sequence $\frac{1}{\# F_n} \sum_{g\in F_n} \varphi(gx)$ converges uniformly in $x$ to $\int_X \varphi\: d\mu$ as $n\to \infty$.
\end{thm}
A dynamical system $(Y,G)$ is called \textbf{equicontinuous} if for all $\varepsilon>0$ there exists $\delta>0$ such that for all $x,y\in Y$ one has $d(x,y)<\delta \implies \forall g\in G: d(gx,gy)<\varepsilon$.
It is well-known that minimal equicontinuous systems are always uniquely ergodic (even for non-amenable $G$).
Furthermore, every dynamical system $(X,G)$ has a so-called \textbf{maximal equicontinuous factor (MEF)} $(X_{\mathrm{eq}},G)$ which is unique up to conjugacy \cite{ElGo1960Homomorphisms}.
Let $\pi_{\mathrm{eq}}:X\to X_{\mathrm{eq}}$ denote the corresponding factor map.
If $\pi_{\mathrm{eq}}^{-1}\{\pi_{\mathrm{eq}}(x)\} = \{x\}$, we call the point $x \in X$ \textbf{almost automorphic}.
Let $Y$ denote the MEF of $\ol{O_G(x)}$ with factor map $\pi$.
It is easy to see that $x$ is almost automorphic if and only if $\pi^{-1}\{\pi(x)\} = \{x\}$.
Thus, if we want to specify the MEF of $\ol{O_G(x)}$, we call $x$ \textbf{almost automorphic over $(Y,G)$} (or just over $Y$).\\
We are predominantly interested in subsystems of the full 0-1 shift.
It is defined as the space $\{0,1\}^G$ equipped with the (metrizable) product topology (where $\{0,1\}$ is endowed with the discrete topology).
We refer to the elements of $\{0,1\}^G$ as ($G$-)\textbf{arrays}.
The natural $G$-action on $\{0,1\}^G$ is the \textbf{(left) shift} which is defined via
\[ (gx)(h) = x(hg) \text{ for } g,h,\in G \text{ and } x\in \{0,1\}^G.\]
A compact and shift-invariant subset $X\tm \{0,1\}^G$ is called a \textbf{subshift}.
Restricting the $G$-action to $X$ yields a dynamical system, which one also calls a subshift.\\
Another class of dynamical systems that will play a role later originates from so-called \textbf{metric group compactifications} $(H,\tau)$ of $G$ which consist of a compact, metrizable group $H$ and an injective group homomorphism $\tau:G\to H$ with dense range.
In this case, $H$ admits a bi-invariant metric $d_H$ (see \cite[Prop.\ 8.43]{HoMo2020CompactGroups}).
Clearly, the dynamical system $(H,\rho,G)$ defined by $\rho: G\times H\to H,\: (g,\xi)\mapsto \tau(g)\cdot \xi$ is minimal and equicontinuous.
The unique invariant measure is given by the Haar measure on $H$.

\subsection{Topological entropy and amorphic complexity}\label{sec: entropy and ac}
Topological entropy and amorphic complexity are both ways to measure the complexity or chaoticity of a dynamical system.
Even though one can define these quantities for arbitrary dynamical systems, we decide -- for the sake of simplicity -- to restrict to subshifts here.
For further reading see \cite{Do2011Entropy} for entropy, and \cite{FuGrJa2016Amorphic, FuGr2020Constant, FuGrJaKw2023Amorphic} for amorphic complexity.\\
We shall first introduce the concept of \textbf{box dimension}.
Let $(Z,d_Z)$ be a metric space and let $\varepsilon>0$.
We say that a subset $A\tm Z$ is \textbf{$\varepsilon$-separated} if $d_Z(z,z^\prime)\geq \varepsilon$ holds for all $z,z^\prime\in A$.
Let $\mathrm{Sep}(Z,d_Z,\varepsilon)$ denote the (possibly infinite) maximal cardinality of an $\varepsilon$-separated subset of $Z$.
We can now define the \textbf{upper box dimension} of $(Z,d_Z)$ as
\[ \ol{\mathrm{Dim}}_{\mathrm{Box}}(Z,d_Z) = \limsup_{\varepsilon\to 0} \frac{\log \mathrm{Sep}(Z,d_Z,\varepsilon)}{-\log \varepsilon}.\]
Analogously we can define $\underline{\mathrm{Dim}}_{\mathrm{Box}}(Z,d_Z)$, and whenever both quantities coincide we denote their common value by $\mathrm{Dim}_{\mathrm{Box}}(Z,d_Z)$.\\
Now fix a F\o lner sequence $\mathcal{F}=(F_n)_{n\in\N}$ in $G$.
We define the \textbf{Besicovitch pseudometric} $D_{\mathcal{F}}$ w.r.t.\ the F\o lner sequence $\mathcal{F}$ via
\[ D_{\mathcal{F}}(x,y) =\limsup_{n\to\infty} \frac{1}{\# F_n} \cdot \# \{g\in F_n: x(g)\neq y(g)\} \quad (x,y \in \{0,1\}^G).\]
For $x\in \{0,1\}^G$, we let $[x]_{\mathcal{F}}$ denote the equivalence class of $x$ w.r.t.\ the equivalence relation
\[y\sim z :\iff D_{\mathcal{F}}(y,z)=0.\]
We denote the quotient space by $[\{0,1\}^G]_{\mathcal{F}}$ and the canonical projection by $\pi_{\mathcal{F}}: \{0,1\}^G \to [\{0,1\}^G]_{\mathcal{F}}$.
Note that $D_{\mathcal{F}}$ is a (well-defined) metric on $[\{0,1\}^G]_{\mathcal{F}}$.
We define the \textbf{upper amorphic complexity} of a subshift $(X,G)$ w.r.t.\ $\mathcal{F}$ as
\[ \ol{\mathrm{ac}}_{\mathcal{F}}(X,G) = \ol{\mathrm{Dim}}_{\mathrm{Box}}([X]_{\mathcal{F}}, D_{\mathcal{F}}).\]
Analogously, one can define $\underline{\mathrm{ac}}_{\mathcal{F}}(X,G)$ and $\mathrm{ac}_{\mathcal{F}}(X,G)$.
\begin{rem}
    It should be pointed out that the approach above is not the usual definition of amorphic complexity as given in \cite{FuGrJa2016Amorphic, FuGrJaKw2023Amorphic}.
    However, it is shown in \cite[Prop.\ 3.19]{FuGrJa2016Amorphic} that for symbolic systems both approaches are equivalent.
\end{rem}
Now we can turn to the definition of topological entropy.
Let $A\tm G$.
An element $\mathfrak{W}$ of the set $\{0,1\}^A$ is called an \textbf{$A$-word}.
We say $\mathfrak{W}$ \textbf{occurs} in the array $x\in \{0,1\}^G$, if there exists $h\in G$ such that for all $g\in A$ one has $x(gh) = \mathfrak{W}(g)$.
Similarly, we say that $\mathfrak{W}$ occurs in a subshift $X$ if $\mathfrak{W}$ occurs in some $x\in X$.
Let $\mathcal{B}_A(X)\tm \{0,1\}^A$ denote the set of all $A$-words which occur in $X$.
Then, the \textbf{topological entropy} of the subshift $(X,G)$ is defined as
\[ \htop(X,G) = \lim_{n\to\infty} \frac{\log \#\mathcal{B}_{F_n}(X)}{\# F_n}.\]
It follows from the Ornstein-Weiss Lemma (c.f.\ \cite[Thm.\ 5]{LackaStraszak2018QuasiUniform}) that above limit indeed always exists and is independent of the F\o lner sequence $\mathcal{F}$, justifying the notation.
This is in contrast to amorphic complexity, whose value in general depends on the choice of $\mathcal{F}$ (see \cite{FuGrJaKw2023Amorphic}).
It should be mentioned that it is also well-known that there is a metric $d$ on $\{0,1\}^G$ which generates the product topology and which satisfies $\htop(X,G) = \mathrm{Dim}_{\mathrm{Box}}(X,d)$.
An important property of entropy for symbolic systems, is the fact that it is continuous w.r.t.\ the \textbf{Weyl pseudometric} $D_{\mathrm{Weyl}}$ which is defined for $x,y\in \{0,1\}^G$ via $D_{\mathrm{Weyl}}(x,y) = \sup_{\mathcal{F}} D_{\mathcal{F}}(x,y)$, where the supremum is taken over all F\o lner sequences in $G$.
\begin{thm}[{\cite[Thm.\ 32]{LackaStraszak2018QuasiUniform}}]\label{thm: entropy cts wrt Weyl}
    The map $\{0,1\}^G \to [0,\infty), \: x \mapsto \htop(\ol{O_G(x)},G)$ is continuous w.r.t.\ $D_{\mathrm{Weyl}}$.
\end{thm}

\section{Residually finite groups and odometers}\label{sec: odometers}
We let $1_G$ denote the neutral element of the group $G$
We call $G$ \textbf{residually finite} if for all $g\in G\setminus \{1_G\}$ there exists a finite group $F$ and a group homomorphism $\varphi: G\to F$ such that $\varphi(g)\neq 1_F$.
Since we assume $G$ to be countable, it is easy to see that $G$ is residually finite if and only if there exists a decreasing sequence $(\Gamma_n)_{n\in\N}$ of finite index normal subgroups $\Gamma_n$ of $G$ such that one has $\bigcap_{n\in\N}\Gamma_n = \{1_G\}$.
Prominent examples of countably infinite, residually finite groups include free groups (on countably many generators) and $\Z^d$.\\
Now let $(\Gamma_n)_{n\in\N}$ be a decreasing sequence of finite index subgroups of $G$ (not necessarily normal) such that $\bigcap_{n\in\N}\Gamma_n = \{1_G\}$.
Then, we can define the associated \textbf{$G$-odometer} $\overleftarrow{G}$ as
\[ \overleftarrow{G} = \setv{(g_n\Gamma_n)_{n\in\N}\in \prod_{n\in\N}G/\Gamma_n}{\forall n\in\N: g_{n+1}\Gamma_n = g_n\Gamma_n}.\]
In other words, if one lets $\phi_n: G/\Gamma_{n+1}\to G/\Gamma_n$ denote the canonical projection, $\overleftarrow{G}$ is then the inverse limit of the inverse system $((G/\Gamma_n)_{n\in\N},(\phi_n)_{n\in\N})$.
Note that the set $\prod_{n\in\N}G/\Gamma_n$ equipped with the product topology (of the discrete topologies on the factors) turns into a compact, metrizable, topological space and $\overleftarrow{G}$ into a compact subspace.
The map $\tau: G \to \overleftarrow{G},\: g\mapsto (g\Gamma_n)_{n\in\N}$ is clearly injective with dense range.
If for $g\in G$ and $\xi = (g_n\Gamma_n)_{n\in\N}$ we define $\rho_g(\xi) = (gg_n\Gamma_n)_{n\in\N}$, then $(\overleftarrow{G},\rho,G)$ is a minimal, equicontinuous system.\\
Furthermore, whenever each subgroup $\Gamma_n$ is normal, each quotient $G/\Gamma_n$ is again a group.
Hence, $\prod_{n\in\N} G/\Gamma_n$ equipped with pointwise composition is a group as well and it is straightforward to verify that $\overleftarrow{G}$ is a compact subgroup.
Thus, $(\overleftarrow{G},\tau)$ is a metric group compactification of $G$ in this case.
For further reading, see \cite{CortezPetite2008Odometers}\\[2pt]
It turns out that for later constructions a different representation of odometers is convenient.
To achieve this alternative representation, we first need a suitable sequence of finite subsets of $G$.
Recall that for a (normal) subgroup $\Gamma$ of $G$ one calls a subset $D\tm G$ a \textbf{fundamental domain} of $G/\Gamma$ if the mapping $D\to G/\Gamma,\: g\mapsto g\Gamma$ is a bijection.
We now assume that each subgroup $\Gamma_n$ is normal.
The following is a straightforward consequence of \cite[Lem.\ 5]{CortezPetite2014Invariant}.
\begin{lem}\label{lem: good fundamental domains}
    Let $G$ be a countably infinite, residually finite, amenable group and let $(\Gamma_n)_{n\in\N}$ be a decreasing sequence of finite index normal subgroups of $G$ such that $\bigcap_{n\in\N} \Gamma_n = \{1_G\}$.
    Then, there exists a strictly increasing sequence $(n_k)_{k\in\N}$ of positive integers and a left F\o lner sequence $(D_k)_{k\in\N}$ in $G$ such that
	\begin{enumerate}
		\item $1_G\in D_1\tm D_2\tm\ldots$ and $D_k$ is a fundamental domain of $G/\Gamma_{n_k}$ for all $k\in\N$.
		\item $\bigcup_{k\in\N} D_k = G$.
		\item $D_l = \biguplus_{h\in D_l \cap \Gamma_{n_k}} D_k h$ for all $l>k$.
	\end{enumerate}
\end{lem}

It follows from \cite[Lem.\ 2]{CortezPetite2008Odometers} that the odometer which is associated to the subsequence $(\Gamma_{n_k})_{k\in\N}$ is isomorphic to $\overleftarrow{G}$ as a topological group.
Thus, we can assume w.l.o.g.\ that $n_k = k$ for all $k\in\N$.\\
The aforementioned alternative representation of $\overleftarrow{G}$ is heavily inspired by the classical "carry-over rule" for adding machines.
We shall only outline how to obtain this representation, while the details can be found in \cite[Section 3.2]{CortezDrewloGomezJager2025Model}.
Observe that the fact that each $D_n$ is a fundamental domain for $G/\Gamma_n$, every $g\in G$ has a unique factorization $g = \psi_n(g)\varphi_n(g)$ with $\psi_n(g)\in D_n$ and $\varphi_n(g)\in \Gamma_n$.
Moreover, iterating this factorization and using (2) of above lemma, one can find a unique positive integer $N(g)$ such that
\[ g = \prod_{j=1}^{N(g)+1} b_j ,\]
with $b_j \in D_j\cap \Gamma_{j-1}$ for all $j=1,\ldots, N(g)+1$.
Here we use the convention $\Gamma_0 = G$ and $\prod_{j=1}^N b_j = b_1 b_2\ldots b_N$.
This composition is furthermore unique and given by $b_j = \varphi_{j-1}(\psi_j(g))$ (see \cite[Lem.\ 3.5]{CortezDrewloGomezJager2025Model}).
We now define $\pi_j: \overleftarrow{G}\to D_j\cap \Gamma_{j-1},\: (g_n\Gamma_n)_{n\in\N} \mapsto \varphi_{j-1}(\psi_j(g_j))$.
It is not hard to see that $\pi_j$ is well-defined and the mapping
\[\Pi: \overleftarrow{G}\to \prod_{n\in\N}(D_n\cap \Gamma_{n-1}),\: \xi \mapsto (\pi_n(\xi))_{n\in\N}\]
is a bijection (c.f.\ \cite[Lem.\ 3.7 and Rem.\ 3.11]{CortezDrewloGomezJager2025Model}).
The space $\overleftarrow{G_*}=\prod_{n\in\N}(D_n\cap\Gamma_{n-1})$ is exactly the alternative representation we sought for.
It follows from \cite[Lem.\ 3.7(2)]{CortezDrewloGomezJager2025Model} that $\Pi$ maps cylinder sets of level $n$ in $\overleftarrow{G}$ to cylinder sets of level $n$ in $\overleftarrow{G_*}$.
In particular, $\Pi$ is a homeomorphism if one equips $\overleftarrow{G_*}$ with the product topology.
Moreover, observe that the embedding $\tau_* = \Pi\circ \tau: G \to \overleftarrow{G_*}$ satisfies for all $\xi\in \overleftarrow{G_*}$ that
\begin{equation}\label{eq: chara image of tau_*}
    \xi \in \tau_*(G) \text{ if and only if } \xi = (\xi_n)_{n\in\N}\text{ is eventually constant } 1_G.
\end{equation}
It only remains to clarify how the group action of $\overleftarrow{G}$ is translated via $\Pi$.
It turns out that for our purposes, we need not know the detailed formula on how to compute the product $\zeta = \xi \cdot \eta$ for two elements $\xi,\eta\in \overleftarrow{G_*}$.
For us, the following consequence of \cite[Lem.\ 3.8]{CortezDrewloGomezJager2025Model} suffices.
\begin{lem}\label{lem: carry-over rule light version}
    Let $\xi=(\xi_n)_{n\in\N}, \eta=(\eta_n)_{n\in\N}\in \overleftarrow{G_*}$ and let $\xi\cdot \eta = \zeta = (\zeta_n)_{n\in\N}$, where $\cdot$ denotes the group action on $\overleftarrow{G_*}$ inherited from $\overleftarrow{G}$.
    Then, we have for all $n\in\N$
    \begin{enumerate}
        \item $\zeta_n$ only depends on $\xi_m$ and $\eta_m$ with $m\leq n$.
        \item If for all $m<n$ we have $\xi_m = 1_G$ or $\eta_m= 1_G$, then $\zeta_n^{-1}(\xi_n\eta_n)\in \Gamma_n$.
    \end{enumerate}
\end{lem}
It turns out that for certain arguments, one representation is more convenient than the other.
In what follows, $\overleftarrow{G}$ will denote the inverse limit representation and $\overleftarrow{G_*}$ the alternative (\textit{carry-over}) representation of an odometer.
Whenever the representation does not matter to us, we will default to the notation $\overleftarrow{G}$.

\section{Toeplitz subshifts}\label{sec: Toeplitz}
A particularly interesting class of symbolic systems, which has been studied extensively in the context of residually finite groups, is the class of Toeplitz subshifts.
For simplicity we restrict ourselves to subshifts over the alphabet $\{0,1\}$, but all considerations remain true for larger alphabets (see \cite{CortezPetite2008Odometers} for further reading).
We define for an array $x$, a subgroup $\Gamma$ of $G$ and a letter $\alpha \in \{0,1\}$ the set $\Per(x,\Gamma,\alpha) = \{g\in G\,|\, \forall \gamma\in \Gamma: x(g\gamma) = \alpha\}$.
Moreover, we let $\Per(x,\Gamma) = \bigcup_{\alpha\in \{0,1\}} \Per(x,\Gamma,\alpha)$.
If for all $g\in G$, there exists a finite index subgroup $\Gamma$ such that $g\in \Per(x,\Gamma)$, we call the array $x$ a \textbf{Toeplitz array} and its orbit closure $\ol{O_G(x)}$ a \textbf{Toeplitz subshift}.
We call a subgroup $\Gamma$ a \textbf{period} of $x$ if $\Per(x,\Gamma)\neq \emptyset$.
If additionally for all $g\in G$ the following implication holds
\[ (\forall\alpha\in \{0,1\}:\: \Per(x,\Gamma,\alpha)\tm \Per(gx,\Gamma,\alpha)) \implies g\in \Gamma,\]
we call the period $\Gamma$ \textbf{essential}.
Let $\mathbf{\Gamma}=(\Gamma_n)_{n\in\N}$ a decreasing sequence of finite index subgroups.
We say that $x$ is \textbf{relative to $\mathbf{\Gamma}$}, if $G = \bigcup_{n\in\N}\Per(x,\Gamma_n)$.
If additionally each $\Gamma_n$ is an essential period of $x$, we call $\mathbf{\Gamma}$ a \textbf{period structure} for $x$.
We denote the set of Toeplitz arrays which are relative to $\mathbf{\Gamma}$ by $\mathrm{Rel}(\mathbf{\Gamma})$ and similarly the set of Toeplitz arrays with period structure $\mathbf{\Gamma}$ by $\mathrm{Toep}(\mathbf{\Gamma})$.
Let $\overleftarrow{G}$ be the $G$-odometer associated to $\mathbf{\Gamma}$.
\begin{prop}[{\cite[Thm.\ 2 and Prop.\ 5]{CortezPetite2008Odometers}}]\label{prop: period structure almost automorphic}
    The set $\mathrm{Toep}(\mathbf{\Gamma})$ is exactly the set of arrays which are almost automorphic over $\overleftarrow{G}$.
\end{prop}

For a finite index subgroup $\Gamma$ of $G$ we let $D(x,\Gamma) = \frac{\# (D\cap \Per(x,\Gamma))}{\# D}$, where $D$ is a fundamental domain of $G/\Gamma$.
Note that this definition is indeed independent of the choice of $D$.
Moreover, if $\Gamma^\prime \tm \Gamma$ is another subgroup of finite index, then it is not hard to see that one can choose a fundamental domain $D$ of $G/\Gamma$ and a fundamental domain $D^\prime$ of $G/\Gamma^\prime$ such that
\[ D^\prime = \bigcup_{\gamma\in D^\prime \cap \Gamma} D\gamma\]
(c.f.\ \cite[Lem.\ 3]{CortezPetite2014Invariant}).
Then, it easily follows that $D(x,\Gamma^\prime)\geq D(x,\Gamma)$ (see \cite{CecchiCortezGomez2024Invariant}).
We now define $\mathcal{D}(x) = \sup\{D(x,\Gamma) : \Gamma \text{ finite index subgroup of } G\} \in (0,1]$.
If $\mathcal{D}(x)=1$, we call the Toeplitz array $x$ \textbf{regular} and otherwise \textbf{irregular}.
Below, we will establish more convenient ways of computing $\mathcal{D}(x)$, for example when given a period structure $\mathbf{\Gamma}$ of $x$.
The fact that those approaches indeed yield the quantity $\mathcal{D}(x)$ as defined above, is folklore knowledge.
But since we could not find a proof of this, we choose to include it for the convenience of the reader.
\begin{lem}[{c.f.\ \cite[Lem.\ 6]{CortezPetite2008Odometers}}]\label{lem: essential period with same periodic points}
    Let $\Gamma$ be a period $x$ and assume that $\Gamma$ is normal.
    Then there exists some essential period $\widetilde{\Gamma}\mt \Gamma$ of $x$ such that $\Per(x,\widetilde{\Gamma},\alpha) = \Per(x,\Gamma,\alpha)$ for all $\alpha$.
    In particular, $D(x,\widetilde{\Gamma}) = D(x,\Gamma)$.
\end{lem}
\begin{proof}
    In the proof of \cite[Lem.\ 6]{CortezPetite2008Odometers}, the authors pick a normal subgroup $\Gamma^\prime$ of $G$ which is contained in $\Gamma$ and construct an essential period $K$.
    Since we assume $\Gamma$ to be normal, we can choose $\Gamma^\prime = \Gamma$.
    Following the proof, one sees that in this case one has $\widetilde{\Gamma} := K \mt \Gamma$, and the remaining properties easily follow.\\
    Alternatively one can prove the lemma with a straightforward application of Zorn's Lemma.
\end{proof}

\begin{prop}\label{prop: regularity constant}
    \hspace{.5cm}
    \begin{enumerate}
        \item We have $\mathcal{D}(x) = \sup\{D(x,\Gamma) : \Gamma \text{ is an essential period of } x\}$.
        \item Let $\mathbf{\Gamma} = (\Gamma_n)_{n\in\N}$ be a period structure for $x$.
        Then $\mathcal{D}(x) = \lim_{n\to\infty}D(x,\Gamma_n)$.
        \item Now assume that $x$ is relative to $\mathbf{\Gamma}=(\Gamma_n)_{n\in\N}$ but all the subgroups $\Gamma_n$ are normal.
        Then we still have $\mathcal{D}(x) = \lim_{n\to\infty} D(x,\Gamma_n)$.
    \end{enumerate}
\end{prop}
\begin{proof}
    (1) Let $\mathcal{D}^\prime(x) = \sup\{D(x,\Gamma) : \Gamma \text{ is an essential period of } x\}$.
    Clearly, we have $\mathcal{D}^\prime(x)\leq D(x)$.
    Now let $\Gamma$ be a finite index subgroup of $G$.
    We can pick some $\Gamma^\prime \tm \Gamma$, which is a finite index normal subgroup of $G$.
    Lemma~\ref{lem: essential period with same periodic points} yields some essential period $\widetilde{\Gamma}$ such that $D(x,\Gamma)\leq D(x,\Gamma^\prime) = D(x,\widetilde{\Gamma}) \leq \mathcal{D}^\prime(x)$.\\
    (2) Note that for any essential period $\Gamma$ of $x$, there exists a period structure $(\Gamma_n^\prime)_{n\in\N}$ such that $\Gamma_1^\prime = \Gamma$ (c.f.\ \cite[Cor.\ 6]{CortezPetite2008Odometers}).
    It follows from \cite{CortezPetite2008Odometers} that for any $n\in \N$ there exists $k_n\in \N$ such that $\Gamma_{k_n}\tm \Gamma_n^\prime$.
    This shows $D(x,\Gamma) \leq \lim_{n\to\infty} D(x,\Gamma_n^\prime) \leq \lim_{n\to\infty} D(x,\Gamma_n)$ and thus $\mathcal{D}^\prime(x) \leq \lim_{n\to\infty} D(x,\Gamma_n)$.
    The converse inequality is obvious.\\
    (3) Lemma~\ref{lem: essential period with same periodic points} gives us a sequence $(\widetilde{\Gamma}_n)_{n\in\N}$ of essential periods such that $\Per(x,\widetilde{\Gamma}_n) = \Per(x,\Gamma_n)$ and $D(x,\widetilde{\Gamma}_n) = D(x,\Gamma_n)$.
    Since
    \[\Per(x,\widetilde{\Gamma}_n) = \Per(x,\Gamma_n) \tm \Per(x,\Gamma_{n+1}) = \Per(x,\widetilde{\Gamma}_{n+1})\]
    it follows from \cite[Lem.\ 5]{CortezPetite2008Odometers} that $\widetilde{\Gamma}_{n+1}\tm \widetilde{\Gamma}_n$.
    Hence, $(\widetilde{\Gamma}_n)_{n\in\N}$ is a period structure for $x$ and by (2) we obtain $\lim_{n\to\infty} D(x,\Gamma_n) = \lim_{n\to\infty} D(x,\widetilde{\Gamma}_n) = \mathcal{D}(x)$.
\end{proof}

\section{Symbolic model sets}\label{sec: model}
Let $(H,\tau)$ be a metric group compactification of $G$ with normalized Haar measure $\nu$.
We write $\mathfrak{B}(H)$ for the (Borel-) measurable subsets of $H$.
Moreover,
{\allowdisplaybreaks[4]
\begin{align*}
    \mathfrak{B}_{\mathrm{gen}}(H) &= \{W\in \mathfrak{B}(H) \,|\, \partial W \cap \tau(G) = \emptyset\}\\
    \mathfrak{B}_{\mathrm{reg}}(H) &= \{W\in \mathfrak{B}(H) \,|\, \nu(\partial W) = 0\}\\
    \mathfrak{B}_{\mathrm{closed}}(H) &= \{W\in \mathfrak{B}(H) \,|\, W \text{ is closed}\}\\
    \mathfrak{B}_{\mathrm{prop}}(H) &= \{W\in \mathfrak{B}(H) \,|\, W \text{ is compact and } W = \ol{\mathrm{int}(W)}\}\\
    \mathfrak{B}_{\mathrm{irred}}(H) &= \{W\in \mathfrak{B}(H) \,|\, \forall\xi\in H: W\xi = W \implies \xi = 1_H\}\\
    \mathfrak{B}_{\mathrm{aper}}(H) &= \{W\in \mathfrak{B}(H) \,|\, \forall\xi\in H: \nu(W\symdiff W\xi) = 0 \implies \xi = 1_H\}\\
    \mathcal{W}(H) &= \mathfrak{B}_{\mathrm{gen}}(H) \cap\mathfrak{B}_{\mathrm{prop}}(H)\cap\mathfrak{B}_{\mathrm{irred}}(H).\\
    \mathcal{W}_{\mathrm{reg}}(H) &= \mathcal{W}(H)\cap \mathfrak{B}_{\mathrm{reg}}(H).
\end{align*}}
In this setting we usually refer to (measurable) subsets $W$ of $H$ as \textbf{windows}.
We call windows $W\in \mathfrak{B}_{\mathrm{gen}}(H)$ (resp.\ $\mathfrak{B}_{\mathrm{reg}}(H)$, $\mathfrak{B}_{\mathrm{prop}}(H)$, $\mathfrak{B}_{\mathrm{irred}}(H)$, $\mathfrak{B}_{\mathrm{aper}}(H)$) \textbf{generic} (resp.\ \textbf{regular}, \textbf{proper}, \textbf{irredundant}, \textbf{Haar aperiodic}).\\
It is clear that any window $W\tm H$ defines an array $x_W = \1_W\circ \tau \in \{0,1\}^G$.
Let us denote this mapping by $\Phi: \mathfrak{B}(H)\to \{0,1\}^G,\: W \mapsto x_W$.
\begin{rem}\label{rem: CPS}
    This situation can be understood as a special form of a so-called \textbf{cut and project scheme (CPS)}.
    A CPS is a triple $(G,H,\mathcal{L})$ consisting of two locally compact groups and an \textbf{irrational lattice} $\mathcal{L}\tm G\times H$.
    See for example \cite{BaLeMo2007Model,BjHaPo2018AperiodicOrder,DrJaLe2025Model} for further reading on general cut and project schemes and \cite{CortezDrewloGomezJager2025Model} for an exposition of the special case from above.
    In our setting, the lattice $\mathcal{L}$ is the graph of the map $\tau$, i.e.\ $\mathcal{L}=\{(g,\tau(g)) : g\in G\}$.
    Furthermore, given a window $W$, one then calls the point set $\{g\in G: x_W(g) = 1\} = \{g\in G: \tau(g)\in W\}$ a \textbf{model set}.
    Hence, we shall refer to $x_W$ as a \textbf{symbolic model set}.\\
    In the literature, one also finds the terms \textbf{topologically regular} instead of proper, and \textbf{topologically aperiodic} instead of irredundant (see \cite{KeRi2019WeakModel}).
\end{rem}
For suitable windows $W$ one can determine the maximal equicontinuous factor of $(\ol{O_G(x_W)},G)$.
In our situation, it turns out that this factor is exactly $(H,\rho,G)$.
\begin{thm}[{\cite[Prop.\ 5.2]{CortezDrewloGomezJager2025Model}}]\label{thm: torus parametrisation}
    Let $W\tm H$ be proper and irredundant.
    Then there exists a unique factor map $\beta: \ol{O_G(x_W)} \to H$ with $\beta(x_W)=1_H$.
    This map satisfies for $y\in \ol{O_G(x_W)}$ and $\xi \in H$ that
    \begin{equation}\label{eq: torus parametrisation}
        \beta(y) = \xi \iff \tau(G)\cap \mathrm{int}(W)\xi^{-1}\tm \{\tau(g): y(g)=1\} \tm \tau(G)\cap W\xi^{-1}.
    \end{equation}
    If additionally $W$ is generic, then $\beta^{-1}\{1_H\} = \{x_W\}$, so that $x_W$ is almost automorphic over $H$.
\end{thm}
One calls the factor map $\beta$ the \textbf{torus parametrisation} for $W$.
We can also use the fact that the array $y$ only takes values 0 or 1 to rewrite \eqref{eq: torus parametrisation} to
\[ \beta(y) = \xi \iff \Phi(\mathrm{int}(W)\xi^{-1}) \leq y \leq \Phi(W\xi^{-1}) \iff x_{\mathrm{int}(W)\xi^{-1}}\leq y \leq x_{W\xi^{-1}}.\]
\begin{thm}[{\cite[Thm.\ 1.1]{CortezDrewloGomezJager2025Model}}]\label{thm: 1-1 windows almost automorphic}
    The map $\Phi$ induces a one-to-one correspondence between $\mathcal{W}(H)$ and the set of arrays $x\in \{0,1\}^G$ which are almost automorphic over $H$.
\end{thm}
Let $\mathbf{\Gamma}=(\Gamma_n)_{n\in\N}$ a sequence of finite index normal subgroups and $\overleftarrow{G}$ be the associated $G$-odometer, which is a metric group compactification of $G$ in this case.
Theorem~\ref{thm: 1-1 windows almost automorphic} and Proposition~\ref{prop: period structure almost automorphic} immediately give a one-to-one correspondence between $\mathcal{W}(\overleftarrow{G})$ and $\mathrm{Toep}(\mathbf{\Gamma})$.
Furthermore, we even have the following consequence of \cite[Thm.\ 5.5]{CortezDrewloGomezJager2025Model}.
\begin{thm}\label{thm: 1-1 regular windows regular Toeplitz}
    The map $\Phi$ induces a one-to-one correspondence between $\mathcal{W}_{\mathrm{reg}}(\overleftarrow{G})$ and the set of regular Toeplitz 0-1 arrays $x\in \mathrm{Toep}(\mathbf{\Gamma})$.
\end{thm}

If one closely inspects the arguments leading to \cite[Thm.\ 1.1]{CortezDrewloGomezJager2025Model}, one finds that irredundancy is a standing assumption for the windows.
However, it is not hard to see one still obtains similar results even after dropping this assumption.
For convenience of the reader, we include the argument (which is also very similar to the proof of \cite[Thm.\ 1]{BaakeJagerLenz2016ToeplitzAndModel}).
\begin{prop}\label{prop: relative Toeplitz generic windows}
	The map $\Phi: \mathfrak{B}_{\mathrm{gen}}(\overleftarrow{G}) \to \mathrm{Rel}(\mathbf{\Gamma}), \: W\mapsto x_W$ is surjective and the restriction of $\Phi$ to $\mathfrak{B}_{\mathrm{gen}}(\overleftarrow{G})\cap \mathfrak{B}_{\mathrm{prop}}(\overleftarrow{G})$ is bijective.
	Moreover, for $W \in \mathfrak{B}_{\mathrm{gen}}(\overleftarrow{G})\cap \mathfrak{B}_{\mathrm{prop}}(\overleftarrow{G})$ we have
	\[ \mathcal{D}(x_W) = 1-\nu(\partial W),\]
	where $\nu$ is the Haar measure on $\overleftarrow{G}$.
	In particular, the map $\Phi: \mathfrak{B}_{\mathrm{gen}}(\overleftarrow{G})\cap \mathfrak{B}_{\mathrm{prop}}(\overleftarrow{G})\cap\mathfrak{B}_{\mathrm{reg}}(\overleftarrow{G}) \to \{x\in \mathrm{Rel}(\mathbf{\Gamma}) : x \text{ is regular}\}$ is bijective as well.
\end{prop}
\begin{proof}
	Let $W \in \mathfrak{B}_{\mathrm{gen}}(\overleftarrow{G})$.
	Let $g\in G$ and assume w.l.o.g.\ that $x_W(g) = 1$, so that $\tau(g) \in W$.
	We aim to show that $g\in \Per(x_W, \Gamma_n, 1)$ for some $n\in\N$.
	Since $W$ is generic, we have $\tau(g)\notin \partial W$, so that $\tau(g)\in \mathrm{int}(W)$.
	Hence, there exists $\varepsilon>0$ such that $B_{\varepsilon}(\tau(g))\tm W$.
	However, $B_{\varepsilon}(\tau(g))$ contains $\tau(g\Gamma_n)$ with suitable $n=n(\varepsilon)\in \N$.
	Therefore, we have $\tau(g\Gamma_n) \tm W$, which yields $x_W(g\gamma_n) = 1$ for all $\gamma_n\in\Gamma_n$, as desired.
	This shows that $x_W\in \mathrm{Rel}(\mathbf{\Gamma})$.\\
	It follows easily using denseness of $\tau(G)$ that the restriction of $\Phi$ to $\mathfrak{B}_{\mathrm{gen}}(\overleftarrow{G})\cap \mathfrak{B}_{\mathrm{prop}}(\overleftarrow{G})$ is injective.
	For surjectivity, let $x\in\mathrm{Rel}(\mathbf{\Gamma})$.
	Define
	\[ U_n = \bigcup \{[\tau(g)]_n : g \in \Per(x,\Gamma_n, 1) \text{ and } V_n = \bigcup \{[\tau(g)]_n : g\in \Per(x,\Gamma_n,0)\}\]
	as well as $U = \bigcup_{n\in\N} U_n$, $V = \bigcup_{n\in\N} V_n$ and $W = \ol{U}$.
	It is straightforward to check that $W$ is generic, proper and $x_W = x$.
	Moreover, it as also not hard to see that $\partial W = \overleftarrow{G}\setminus (U\cup V)$.\\
	Now observe that with $D_n$ a fundamental domain of $G/\Gamma_n$ one has
	\[ \nu(U_n\cup V_n) = \frac{1}{\# D_n} \cdot \#\kl D_n \cap \kl\Per(x_W,\Gamma_n,1)\cup \Per(x_W,\Gamma_n,0)\kr\kr = D(x_W,\Gamma_n).\]
	Hence, by Proposition~\ref{prop: regularity constant} we obtain
	\[ \mathcal{D}(x_W) = \lim_{n\to\infty} D(x_W,\Gamma_n) = \lim_{n\to\infty} \nu(U_n\cup V_n) = \nu(U\cup V) = 1- \nu(\partial W). \qedhere\]

%    \begin{align*}
%        \Per(x_W,\Gamma_n,1) &= \{g\in G: [\tau(g)]_n \tm W\} \text{ and }\\
%        \Per(x_W,\Gamma_n,0) &= \{g\in G : [\tau(g)]_n \cap W = \emptyset\}.
%    \end{align*}
%    Together with the definition of the Haar measure $\nu$, we obtain
%    \begin{equation}\label{eq: 1-D_n as measure}
%        D(x_W,\Gamma_n) = 1- \nu \kl \bigcup \{C\tm \overleftarrow{G}: C \text{ cylinder of level } n \text{ and } C\cap  W \neq \emptyset \neq C \setminus W \} \kr.
%    \end{equation}
%    Since clearly $\bigcap_{n\in\N} \bigcup \{C\tm \overleftarrow{G}: C \text{ cylinder of level } n \text{ and } C\cap  W \neq \emptyset \neq C \setminus W \} = \partial W$, it follows that $\lim_{n\to\infty} D(x_W,\Gamma_n) = 1-\nu(\partial W)$ as desired.
\end{proof}

\bigskip
Now we shall demonstrate that after equipping the spaces $\mathfrak{B}(H)$ and $\{0,1\}^G$ with fairly natural pseudometrics the map $\Phi: \mathfrak{B}(H) \to \{0,1\}^G, \: W \mapsto x_W := \1_W \circ \tau$ becomes continuous.
We define pseudometrics $D$ and $\ol{D}$ on $\mathfrak{B}(H)$ as follows:
$$ D(A,B) =  \nu(A\symdiff B) \quad \text{and} \quad \ol{D}(A,B) = \nu(\ol{A \symdiff B}),$$
so that clearly $D \leq \ol{D}$.

\begin{prop}\label{prop: metric on proper}
	The pseudometrics $D$ and $\ol{D}$ are metrics when restricted to the proper windows in $H$.
\end{prop}
\begin{proof}
	It clearly suffices to show that $D$ is a metric on the proper windows.
	Therefore, let $W \neq W^\prime$ be proper windows, so that $W = \ol{\mathrm{int}(W)}$ and $W^\prime = \ol{\mathrm{int}(W^\prime)}$.
	This implies that $\mathrm{int}(W)\not\tm W^\prime$ or $\mathrm{int}(W^\prime)\not\tm W$.
	Assume w.l.o.g.\ that the former case holds, hence $\mathrm{int}(W)\setminus W^\prime$ is non-empty and open.
	Since $\nu$ is a Haar measure, we obtain
	$$ 0 < \nu(\mathrm{int}(W)\setminus W^\prime) \leq \nu(W\symdiff W^\prime) = D(W,W^\prime). \qedhere$$
\end{proof}

\begin{cor}\label{cor: chara irred for proper}
	Let $W\tm H$ be proper. Then $\mathfrak{B}_{\mathrm{irred}}(H) = \mathfrak{B}_{\mathrm{aper}}(H)$.
\end{cor}
The following result is Theorem~\ref{thm: properties Phi intro}
\begin{thm}\label{thm: properties Phi}
Let $\mathcal{F}=(F_n)_{n\in\N}$ be a F\o lner sequence in $G$ and equip $\{0,1\}^G$ with the Besicovitch pseudometric $D_{\mathcal{F}}$.
\begin{enumerate}
	\item The map $\Phi$ is continuous w.r.t.\ $\ol{D}$ and $D_{\mathcal{F}}$.
	\item The map $\Phi\vert_{\mathfrak{B}_{\mathrm{reg}}(H)}$ is an isometry between $D$ and $D_{\mathcal{F}}$.
\end{enumerate}
Moreover, (1) and (2) also hold with $D_{\mathcal{F}}$ replaced by the Weyl pseudometric $D_{\mathrm{Weyl}}$.
\end{thm}
\begin{proof}
	(1) Let $\varepsilon>0$ and let $W,W^\prime \in \mathfrak{B}(H)$ such that $\nu(\ol{W \symdiff W^\prime})< \varepsilon/2$.
	By Urysohn's Lemma, we can find $\varphi \in C(H)$ such that $\varphi\geq \1_{\ol{W\symdiff W^\prime}}$ and $\int \varphi \: d\nu - \nu(\ol{W\symdiff W^\prime}) < \varepsilon/2$.
	Recall that the unique invariant measure for the $G$-action $\rho: G\times H \to H, \: (g,\xi) \mapsto \rho_g(\xi):= \tau(g)\xi$ is the Haar measure $\nu$.
	Thus, by the uniform ergodic theorem we have
	\begin{align*}
		D_{\mathcal{F}}(x_W, x_{W^\prime}) &= \limsup_{n\to\infty} \frac{1}{\# F_n} \cdot\# \{ g\in F_n: x_W(g) \neq x_{W^\prime}(g)\}\\
		&= \limsup_{n\to\infty} \frac{1}{\# F_n}\cdot \# \{g\in F_n: \tau(g) \in W\symdiff W^\prime\} \\
		&\leq  \limsup_{n\to\infty} \frac{1}{\# F_n} \sum_{g\in F_n} \1_{\ol{W\symdiff W^\prime}} (\rho_g(1_H)) \\
		&\leq \limsup_{n\to\infty} \frac{1}{\# F_n} \sum_{g\in F_n} \varphi(\rho_g(1_H))
		= \int \varphi\: d\nu < \varepsilon
	\end{align*}
	(2)
	Now assume $W,W^\prime \in \mathfrak{B}_{\mathrm{reg}}(H)$.
	We can pick $\varphi_1,\varphi_2\in C(H)$ such that $\varphi_1 \leq \1_{\mathrm{int}(W\symdiff W^\prime)} \leq \1_{\ol{W\symdiff W^\prime}} \leq \varphi_2$ as well as
	$$\int \varphi_2\: d\nu - \nu(\ol{W\symdiff W^\prime}) <\varepsilon \text{ and }\nu(\mathrm{int}(W\symdiff W^\prime)) - \int \varphi_1 \: d\nu < \varepsilon.$$
	Analogously to above, the uniform ergodic theorem implies
	$$ \int \varphi_1 \: d\nu \leq  D_{\mathcal{F}}(x_W,x_{W^\prime}) \leq \int \varphi_2 \: d\nu$$
	Now observe that $\partial ( W\symdiff W^\prime) \tm \partial W\cup \partial W^\prime$, so that $D(W,W^\prime) = \nu(W\symdiff W^\prime) = \nu(\mathrm{int}(W\symdiff W^\prime)) = \nu(\ol{W\symdiff W^\prime})$.
	This implies $\abs{D_{\mathcal{F}}(x_W,x_{W^\prime})-D(W,W^\prime)} \leq \varepsilon$.
	Since $\varepsilon>0$ was arbitrary, we conclude that $\Phi\vert_{\mathfrak{B}_{\mathrm{reg}}(H)}$ is an isometry, as desired.\\
	Lastly, observe that all previous estimates and observations can be made uniformly in $\mathcal{F}$, so that (1) and (2) indeed also hold for $D_{\mathrm{Weyl}}$.
\end{proof}

\section{Amorphic Complexity}\label{sec: ac}
We shall first provide an alternative formula for the amorphic complexity of regular symbolic model sets.
The following observations can be found in a slightly different context in \cite[Section~2.3]{KasjanKeller2025Besicovitch}.
Let $\mathcal{F}$ be a F\o lner sequence in $G$ and $W\in \mathcal{W}_{\mathrm{reg}}(H)$.
Let $\beta_W: \ol{O_G(x_W)} \to H$ be the torus parametrisation (c.f.\ Theorem~\ref{thm: torus parametrisation}).
Moreover, let $\pi_W$ be the restriction of the canonical projection $\pi_{\mathcal{F}}:\{0,1\}^G\to [\{0,1\}^G]_{\mathcal{F}}$ to $\ol{O_G(x_W)}$.
We equip $H$ with the metric $d_W$ defined for $\xi,\eta\in H$ via
\[ d_W(\xi,\eta) = D(W\xi^{-1},W\eta^{-1}) = \nu(W\xi^{-1} \symdiff W\eta^{-1}).\]
This is indeed a metric due to properness and irredundancy of $W$, together with Proposition~\ref{prop: metric on proper} and Corollary~\ref{cor: chara irred for proper}.
Moreover, $d_W$ is clearly left-invariant w.r.t.\ the group action of $H$.
\begin{thm}
	There exists a (unique) bijection $\psi_W: H \to [\ol{O_G(x_W)}]_{\mathcal{F}}$ such that $\pi_W(\beta_W^{-1}\{\xi\}) = \{\psi_W(\xi)\}$ for all $\xi \in H$.
	Moreover, $\psi_W$ is an isometry between $d_W$ and $D_{\mathcal{F}}$.
\end{thm}
\begin{proof}
	First, we claim the following:
	Let $\xi\in H$ and  $x,y\in \{0,1\}^G$ (not necessarily $x,y\in \ol{O_G(x_W)}$) such that $x_{\mathrm{int}(W)\xi^{-1}} \leq x,y \leq x_{W\xi^{-1}}$.
	Then, $D_{\mathcal{F}}(x,y) = 0$.
	To see that this is true, note that $x(g) \neq y(g)$ implies $\rho_g(\xi) = \tau(g)\cdot \xi \in \partial W$.
	Now the claim follows similarly to the proof of Theorem~\ref{thm: properties Phi} using the uniform ergodic theorem and Urysohn's Lemma.\\
	The claim readily implies that for any $\xi\in H$, the set $\pi_W(\beta_W^{-1}\{\xi\})$ is a singleton, so that the map $\psi_W$ is well-defined.
	Moreover, $\psi_W$ is clearly surjective so that it remains to show that $\psi_W$ is an isometry.
	Recall that due to Theorem~\ref{thm: properties Phi}(2), the map $H \to \{0,1\}^G,\: \xi\mapsto x_{W\xi^{-1}}$ is an isometry between $d_W$ and $D_{\mathcal{F}}$, whilst $\pi_{\mathcal{F}}$ is an isometry by definition.
	Hence, it suffices to show that $\pi_{\mathcal{F}}(x_{W\xi^{-1}}) = \psi_W(\xi)$.
	(Note that it is in general \textbf{not true} that $x_{W\xi^{-1}} \in \beta_W^{-1}\{\xi\}$.)
	However, above claim yields $D_{\mathcal{F}}(y,x_{W\xi^{-1}}) = 0$ for any $y\in \beta_W^{-1}\{\xi\}$, i.e.\ $\pi_W(y) = \pi_{\mathcal{F}}(x_{W\xi^{-1}})$, finishing the proof.
\end{proof}

\begin{cor}\label{cor: ac is box dim of d_W}
    We have
	\[ \ol{\mathrm{ac}}_{\mathcal{F}}(\ol{O_G(x_W)},G) = \ol{\mathrm{Dim}}_{\mathrm{Box}}(H,d_W) = \limsup_{\varepsilon\to 0} \frac{\log \nu\kl B_{\varepsilon}^{d_W}(1_H)\kr}{\log \varepsilon}.\]
	In particular, the upper amorphic complexity of $(\ol{O_G(x)},G)$ is independent of the F\o lner sequence $\mathcal{F}$.
	An analogous result holds for the lower amorphic complexity, as well.
\end{cor}
The latter equality is a well-known consequence of the (left-)invariance of the metric $d_W$ and the measure $\nu$.
Due to this result, we will write $\ol{\mathrm{ac}}(\ol{O_G(x_W)},G)$ instead of $\ol{\mathrm{ac}}_{\mathcal{F}}(\ol{O_G(x_W)},G)$ for $W\in \mathcal{W}_{\mathrm{reg}}(H)$.

\subsection{Upper estimates}
The first part of this subsection will be dedicated to proving Theorem~\ref{thm: estimate ac Toeplitz intro}.
Furthermore, using a Toeplitz example we show that the estimate given by \cite[Thm.\ 1.3]{FuGrJaKw2023Amorphic} does not hold in the full generality claimed by the authors.\\
Let $\mathbf{\Gamma}=(\Gamma_n)_{n\in\N}$ be a decreasing sequence of normal subgroups of $G$ with finite index such that $\bigcap_{n\in\N} \Gamma_n = \{1_G\}$ and let $(D_n)_{n\in\N}$ be a sequence of fundamental domains of $G/\Gamma_n$ given by Lemma~\ref{lem: good fundamental domains}.
Let $x\in \mathrm{Rel}(\mathbf{\Gamma})$ be a regular Toeplitz array, so that Proposition~\ref{prop: regularity constant} yields $\lim_{n\to\infty} D(x,\Gamma_n) = 1$.
Let $\overleftarrow{G}$ denote the $G$-odometer associated to $\mathbf{\Gamma}$ and let $\nu$ be the (normalized) Haar measure on $\overleftarrow{G}$.
The following is a natural generalisation of \cite[Thm.\ 1.6]{FuGrJa2016Amorphic}.

\begin{thm}\label{thm: estimate ac Toeplitz}
    We have
    $$ \ol{\mathrm{ac}}(\ol{O_G(x)},G) \leq \limsup_{n\to\infty} \frac{\log \# D_{n+1}}{-\log(1-D(x,\Gamma_n))}.$$
\end{thm}
\begin{proof}
    By Proposition~\ref{prop: relative Toeplitz generic windows} we can find a unique proper, generic, regular window $W\tm \overleftarrow{G}$ such that $x = x_W$.
    Let us denote $\varepsilon_n = 1-D(x,\Gamma_n)$.
    We make the following observation: Let $\xi = (d_n\Gamma_n)_{n\in\N} \in \overleftarrow{G}$ and assume that $d_n \in \Gamma_n$.
    Recall the sets $U_n = \bigcup \{[\tau(g)]_n : g \in \Per(x,\Gamma_n, 1)$ and $V_n = \bigcup \{[\tau(g)]_n : g\in \Per(x,\Gamma_n,0)\}$ from the proof of Proposition~\ref{prop: relative Toeplitz generic windows}.
    Then, since $d_n\in \Gamma_n$ we clearly have $U_n \xi^{-1} = U_n$ and $V_n\xi^{-1} = V_n$.
    This yields
    \[ d_W(1_{\overleftarrow{G}},\xi) = \nu(W\symdiff W\xi^{-1}) \leq \nu\kl\overleftarrow{G}\setminus(U_n\cup V_n)\kr = 1 - D(x,\Gamma_n) = \varepsilon_n.\]
%    \begin{align*}
%        d_W(1_{\overleftarrow{G}},\xi) &= \nu(W\symdiff W\xi^{-1}) \\
%        &\leq \nu \kl \bigcup \{C\tm \overleftarrow{G}: C \text{ cylinder of level } n \text{ and } C\cap W \neq \emptyset \neq C\setminus W \} \kr \\
%        \overset{\eqref{eq: 1-D_n as measure}}&{=} \varepsilon_{n}.
%    \end{align*}
    This shows that $\nu\kl B_{\varepsilon}^{d_W}\kl 1_{\overleftarrow{G}}\kr\kr \geq \nu\kl\il 1_{\overleftarrow{G}}\ir_n\kr\geq \frac{1}{\# D_n}$ for $\varepsilon>\varepsilon_{n}$.
    If we define $\psi(\varepsilon) = \frac{1}{\# D_n}$ where $\varepsilon_{n}<\varepsilon\leq \varepsilon_{n-1}$, then we have just shown that
    $$ \frac{\log \nu(B_{\varepsilon}^{d_W}(1_{\overleftarrow{G}}))}{\log \varepsilon} \leq \frac{\log \psi(\varepsilon)}{\log \varepsilon}.$$
    Now, observe that the mapping $\varepsilon \mapsto \frac{\log \psi(\varepsilon)}{\log \varepsilon}$ restricted to the interval $(\varepsilon_{n+1}, \varepsilon_{n}]$ has a maximum at $\varepsilon_{n}$.
    Hence,
    \[
        \limsup_{\varepsilon\to 0}\frac{\log \psi(\varepsilon)}{\log \varepsilon}
        = \limsup_{n\to\infty} \frac{\log \psi(\varepsilon_{n})}{\log \varepsilon_{n}}
        = \limsup_{n\to\infty}\frac{\log \# D_{n+1}}{-\log(1-D(x,\Gamma_n))}
    \]
    Together with Corollary~\ref{cor: ac is box dim of d_W}, this yields that
    \begin{align*}
        \ol{\mathrm{ac}}(\ol{O_G(x)},G) &= \ol{\mathrm{Dim}}_{\mathrm{Box}}(\overleftarrow{G},d_W) = \limsup_{\varepsilon\to 0} \frac{\log \nu(B_{\varepsilon}^{d_W}(1_{\overleftarrow{G}}))}{\log \varepsilon}
        \leq \limsup_{\varepsilon\to 0}\frac{\log \psi(\varepsilon)}{\log \varepsilon}\\
        &= \limsup_{n\to\infty} \frac {\log \# D_{n+1}}{-\log(1-D(x,\Gamma_n))}
    \end{align*}
    finishing the proof.
\end{proof}
Since $\# D_{n+1} = [G:\Gamma_{n+1}]$, this entails Theorem~\ref{thm: estimate ac Toeplitz intro}.
Also note that unlike the proof of \cite[Thm.\ 1.6]{FuGrJa2016Amorphic}, above argument uses the fact that we can represent a Toeplitz array as a model set (recall Remark~\ref{rem: CPS}).\\
In fact, one may deem a proof for Theorem~\ref{thm: estimate ac Toeplitz} obsolete, as it is a consequence of \cite[Thm.\ 1.3]{FuGrJaKw2023Amorphic}, which is a stronger bound for a more general class of model sets.

\begin{namedtheorem}[Claim]{\normalfont (\cite[Thm.\ 1.3]{FuGrJaKw2023Amorphic})}
    Let $G,H$ be locally compact, abelian, second countable groups, let $(G,H,\mathcal{L})$ be a cut and project scheme and let $W\tm H$ be a regular, irredundant window.
    If $\ol{\mathrm{Dim}}_{\mathrm{Box}}(H)$ is finite, then for any F\o lner sequence $\mathcal{F}$ in $G$ one has
    $$ \ol{\mathrm{ac}}_{\mathcal{F}}(\Omega(\Lambda(W)),G) \leq \frac{\ol{\mathrm{Dim}}_{\mathrm{Box}}(H)}{\ol{\mathrm{Dim}}_{\mathrm{Box}}(H)-\ol{\mathrm{Dim}}_{\mathrm{Box}}(\partial W)}$$
    where the (upper) box dimensions are computed w.r.t.\ an arbitrary invariant metric on $H$.
\end{namedtheorem}

\begin{namedtheorem}[Consequence]
Whenever $(H,\tau)$ is a metric group compactification of $G$ and $W\in\mathcal{W}_{\mathrm{reg}}(H)$ such that $\ol{\mathrm{Dim}}_{\mathrm{Box}}(\partial W) = 0$ and $\ol{\mathrm{Dim}}_{\mathrm{Box}}(H)<\infty$ w.r.t.\ some invariant metric on $H$, then $\ol{\mathrm{ac}}(\ol{O_G(x_W)},G)\leq 1$.
\end{namedtheorem}

We shall now construct a counterexample to this consequence.\\[3pt]
Let $q_1 =3$ and define $q_{n+1} = 2^{q_n-1}$ inductively.
For the groups $G$ and $H$ we pick $G = \Z$ and $H = \prod_{n\in\N} \{0,\ldots,q_n-1\}$ the classical odometer (adding machine) associated to the sequence $(q_n)_{n\in\N}$ with carry-over addition.
The group homomorphism $\tau: \Z\to H$ is characterised by $\tau(1) = (1,0,0,\ldots)$.
\begin{prop}\label{prop: counterexample}
    There exists a window $W\in\mathcal{W}_{\mathrm{reg}}(H)$, such that $\ol{\mathrm{ac}}(\ol{O_\Z}(x_W),\Z) = \infty$ and $\ol{\mathrm{Dim}}_{\mathrm{Box}}(\partial W) = 0$ for any metric on $H$.
\end{prop}
\begin{proof}
    For $(\xi_1,\ldots,\xi_k) \in \prod_{n=1}^k\{0,\ldots,q_n-1\}$ we let $[\xi_1,\ldots,\xi_k] := [(\xi_1,\ldots,\xi_k)] \tm H$ denote the corresponding cylinder set.
    Similarly, for $\xi\in H$ and $n\in\N$ we let $[\xi]_n = [\xi_1,\ldots,\xi_n]$.
    We define the window $W$ iteratively:\\
%    \underline{Step 1:}
%    For $l = 0,\ldots,q_1-2$ we let
%    $$ M_1^{(l)} = \{\xi_2 \in \{0,\ldots,q_2-1\}: \xi_2 \mod 2^{l+1} < 2^l \}.$$
%    Now we define $U_1^{(l)} = \bigcup_{\xi_2\in M_1^{(2)}} [(l,\xi_2)]$ and $V_1^{(2)} = [l]\setminus U_1^{(l)}$.
%    Lasly, we let $U_1 = \bigcup_{l=0}^{q_1-2} U_1^{(l)}$ and $V_1 = \bigcup_{l=1}^{q_1-2} V_1^{(l)}$.
%    Note that $H\setminus (U_1\cup V_1) = [q_1-1]$.
%    Also, the fact that $q_2 = 2^{q_1-1}$ readily implies $\# M_1^{(l)} = q_2/2$ so that $\nu(U_1^{(l)}) = \nu(V_1^{(l)}) = \nu([l])/2$.\\
    \underline{Step $n$:}
    For $l= 0,\ldots,q_n-2$ let
    $$ M_n^{(l)} = \{\xi_{n+1} \in \{0,\ldots,q_{n+1}-1\} : (\xi_{n+1}\mod 2^{l+1}) < 2^l\}.$$
    Now we define $U_n^{(l)} = \bigcup_{\xi_{n+1} \in M_n^{(l)}} [q_1-1,\ldots,q_{n-1}-1,l,\xi_{n+1}]$ and $V_n^{(l)} = [q_1-1,\ldots,q_{n-1}-1,l]\setminus U_n^{(l)}$.
    Lastly, we let $U_n = \bigcup_{l=0}^{q_n-2} U_n^{(l)}$ and $V_n = \bigcup_{l=0}^{q_n-2} V_n^{(l)}$.\\
    For illustration, let us compute $U_1$ and $U_2$.
    Note that $q_1 = 3$, $q_2 = 2^{q_1-1} = 4$ and $q_3 = 2^{q_2-1} = 8$.
    This means that $U_1 = U_1^{(0)} \cup U_1^{(1)}$.
    We compute $M_1^{(0)} = \{\xi_2 \in \{0,1,2,3\} : (\xi_2\mod 2) < 1\} = \{0,2\}$, so that $U_1^{(0)} = [0,0] \cup [0,2]$.
    Analogously, it follows that $M_1^{(1)} = \{\xi_2\in \{0,1,2,3\} : (\xi_2 \mod 4) < 2\} = \{0,1\}$, so that $U_1^{(1)} = [1,0] \cup [1,1]$.
    Similarly, we see that $U_2 = \bigcup_{l=0}^2 U_2^{(l)}$ where $U_2^{(0)} = [2,0,0] \cup [2,0,2] \cup [2,0,4] \cup [2,0,6]$, $U_2^{(1)} = [2,1,0] \cup [2,1,1] \cup [2,1,4] \cup [2,1,5]$ and $U_2^{(2)} = [2,2,0]\cup[2,2,1]\cup[2,2,2]\cup[2,2,3]$.
    See Figure~\ref{fig: counterex} for a visualisation.

    \begin{figure}
    \centering
    \includegraphics[width=.7\linewidth]{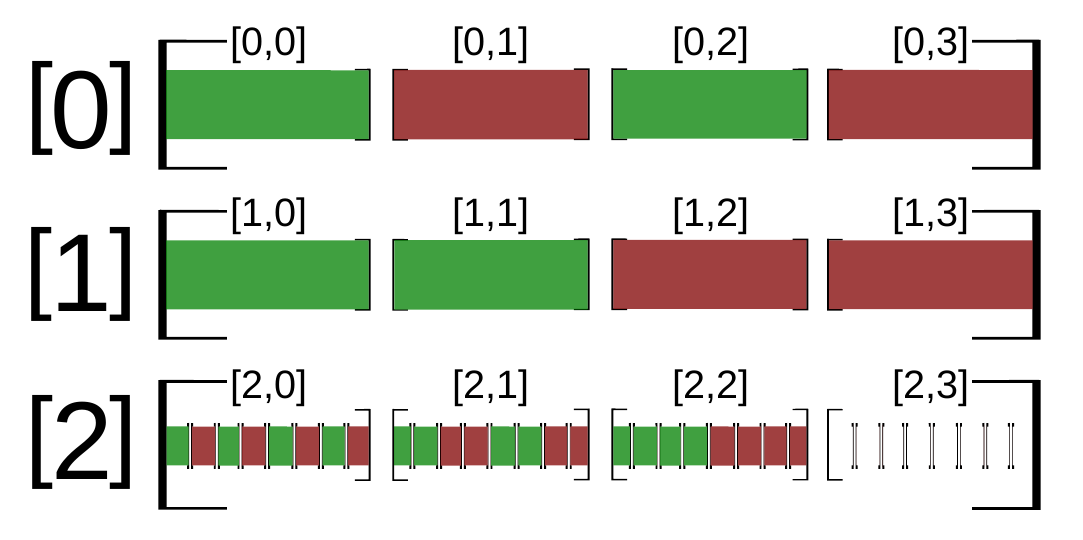}
    \caption{\it The construction of $W$ after 2 steps.
    The green area describes $U_1\cup U_2$ and the red area is $V_1\cup V_2$.}
    \label{fig: counterex}
    \end{figure}

    The fact that $q_{n+1} = 2^{q_n-1}$ readily implies $\# M_n^{(l)} = q_{n+1}/2$ so that
    $$  \nu(U_n^{(l)}) = \nu(V_n^{(l)})= \nu([q_1-1,\ldots,q_{n-1}-1,l])/2.$$
    An easy induction yields $H\setminus \bigcup_{j=1}^n (U_j\cup V_j) = [q_1-1,\ldots,q_n-1]$.\\
    We now define (similarly to the proof of Proposition~\ref{prop: relative Toeplitz generic windows}) $U = \bigcup_{n\in\N} U_n$, $V = \bigcup_{n\in\N} V_n$ and $W = \ol{U}$, so that $W$ is clearly proper.
    Moreover, by above observation we have
    \begin{equation}\label{eq: boundary of W (counterex)}
        \partial W = H\setminus (U\cup V) = \{(q_n-1)_{n\in\N}\} = \{\tau(-1)\},
    \end{equation}
    which shows regularity and $\ol{\mathrm{Dim}}_{\mathrm{Box}}(\partial W) = 0$ for any metric on $H$.\\
    To compute the upper amorphic complexity, let
    $$ \varepsilon_n = \nu([0_H]_n)/2 = \frac{1}{2}\cdot \prod_{j=1}^n q_j^{-1}.$$
    Let $\xi = (\xi_n)_{n\in\N}\in H$.
    Observe that if $\xi_1\neq 0$, then $\nu(W\symdiff (W-\xi))\geq \varepsilon_1$.
    Furthermore, if $\xi_1=0$ and $\xi_2 \neq 0$, then we even have $\nu(W\symdiff(W-\xi)) \geq \nu([0])\geq \varepsilon_1$.
    (This can be seen best visually in Figure~\ref{fig: counterex}.)
    Altogether this shows $B_{\varepsilon_1}^{d_W}(0_H) \tm [0,0]$.
    Since $\varepsilon_2<\varepsilon_1$, we clearly have $B_{\varepsilon_2}^{d_W}(0_H)\tm [0,0]$.
    However, similarly to above if $\xi_1=\xi_2=0$ but $\xi_3\neq 0$, then $\nu(W\symdiff(W-\xi))\geq \varepsilon_2$.
    Thus, $B_{\varepsilon_2}^{d_W}(0_H)\tm [0,0,0]$.
    Inductively we obtain $B_{\varepsilon_n}^{d_W}(0_H) \tm [0_H]_{n+1}$.
    Observe that this also implies irredundancy of $W$.\\
    Since the base of the logarithm does not matter for computing the box dimension, we let $\log$ denote the logarithm with base 2 for this proof.
    It follows that
    \begin{align*}
        \ol{\mathrm{ac}}(\ol{O_\Z(x_W)},\Z) &=\limsup_{\varepsilon\to 0} \frac{\log \nu(B_\varepsilon^{d_W}(0_H))}{\log \varepsilon}
            \geq \limsup_{n\to\infty} \frac{\log \nu(B_{\varepsilon_n}^{d_W}(0_H))}{\log \varepsilon_n}\\
            &\geq \limsup_{n\to\infty} \frac{\sum_{j=1}^{n+1} \log q_j}{\sum_{j=1}^n \log q_j+\log 2}
            \geq \limsup_{n\to\infty} \frac{\log q_{n+1}}{\sum_{j=1}^n \log q_j+ \log 2}\\
            &= \limsup_{n\to\infty} \frac{q_n-1}{\sum_{j=1}^n \log q_j+ 1}
            \geq \limsup_{n\to\infty} \frac{q_n-1}{n\log q_n +1}\\
            &\geq \limsup_{n\to\infty} \frac{q_n-1}{(\log q_n)^2+1}
            =\infty,
    \end{align*}
    where the last inequality follows since $n\leq \log q_n$ holds for all $n\in\N$.\\
    It should be noted that \eqref{eq: boundary of W (counterex)} clearly yields non-genericity of $W$.
    However, it suffices to replace $W$ by any $W^\prime =\zeta + W$ with $\zeta\in H\setminus \tau(\Z)$.
    Then, $W^\prime \in \mathcal{W}_{\mathrm{reg}}(H)$ and $\ol{\mathrm{Dim}}_{\mathrm{Box}}(\partial W^\prime) = 0$.
    Lastly, we have $\ol{\mathrm{ac}}(\ol{O_{\Z}(x_{W^\prime})},\Z) = \ol{\mathrm{ac}}(\ol{O_{\Z}(x_W)},\Z) = \infty$ since the metrics $d_{W^\prime}$ and $d_W$ coincide.
    (In fact, the systems are conjugate.)
    This concludes the proof.
\end{proof}

\begin{rem}
    With fairly straightforward adaptations to the construction one can obtain an analogous example with $\Z$ replaced by any finitely generated, torsion-free, nilpotent group $G$ and $H$ replaced by a suitable $G$-odometer.
    However, it should be mentioned that \cite[Thm.\ 1.3]{FuGrJaKw2023Amorphic} remains valid (with the proof given by the authors) whenever the internal group $H$ is $\R^d$ for some $d\in\N$.
\end{rem}

\subsection{Toeplitz arrays with arbitrary amorphic complexity}\label{sec: krieger/ac}
In this seciton we will give the proof of Theorem~\ref{thm: realize ac Toeplitz intro}.
We assume that $G$ is an amenable, residually finite group such that there exists integers $s\geq 1$, $p\geq 3$ and a decreasing sequence $(\Gamma_n)_{n\in\N}$ of finite index normal subgroups $\Gamma_n$ with $\bigcap_{n\in\N}\Gamma_n = \{1_G\}$ and $[G: \Gamma_n] = sp^n$ for all $n\in\N$.
We will discuss towards the end that this is satisfied for all finitely generated, torsion-free, nilpotent groups.
We let $\overleftarrow{G_*}$ denote the $G$-odometer associated to $(\Gamma_n)_{n\in\N}$ with the carry-over representation (c.f.\ Section~\ref{sec: odometers}).
We see that in view of Corollary~\ref{cor: ac is box dim of d_W}, it suffices to construct a window $W \in \mathcal{W}_{\mathrm{reg}}(\overleftarrow{G_*})$ such that $\mathrm{Dim}_{\mathrm{Box}}(\overleftarrow{G_*},d_W)$ is our arbitrary value.
We do this, by first constructing for each $t\in [0,1]$ a metric $d_t$ on $\overleftarrow{G_*}$ such that $\mathrm{Dim}_{\mathrm{Box}}(\overleftarrow{G_*},d_t)$ ($t\in [0,1]$) takes a large range of values and then construct a window $W$ so that $d_W$ and $d_t$ are Lipschitz-equivalent.\\
According to Section~\ref{sec: odometers}, $\overleftarrow{G_*} =\prod_{n\in\N} (D_n \cap \Gamma_{n-1})$, where $\Gamma_0 = G$ and $(D_n)_{n\in\N}$ is a suitable sequence of fundamental domains given by Lemma~\ref{lem: good fundamental domains}.
Also note that
\begin{equation}\label{eq: size D_n cap Gamma_{n-1}}
	\#(D_n\cap\Gamma_{n-1}) = [\Gamma_{n-1}:\Gamma_n] =\begin{cases} p &, \text{ if } n>1\\
														      sp &, \text{ if } n=1 \end{cases}.
\end{equation}
Recall that the Haar measure $\nu$ satisfies for cylinder sets $[\xi]_m := \prod_{n\leq m} \{\xi_n\} \times \prod_{n>m} (D_n\cap \Gamma_{n-1})$ that
$$\nu([\xi]_m) = \prod_{n\leq m}\frac{1}{\#(D_n\cap \Gamma_{n-1})} =\frac{1}{[G:\Gamma_m]} =  \frac{1}{sp^m}.$$
For the construction of $d_t$, fix $t\in [0,1]$ and let $N(t)$ be a subset of $\N$ such that $\mathrm{dens}(N(t)) := \lim_{n\to\infty} \frac{1}{n} \cdot \# \{1\leq k\leq n : k\in N(t)\}= t$.
We define for $n\in \N$
$$ \varepsilon(t,n) = p^{-\#\{1\leq k\leq n\::\: k\notin N(t)\}} \cdot \kl \frac{p}{p-2}\kr^{-\#\{1\leq k\leq n\::\: k\in N(t)\}}> 0.$$
For $\xi,\eta\in \overleftarrow{G_*}$, we let $d_t(\xi,\eta)=0$ if $\xi = \eta$.
Otherwise, let $n = \min\{k\in\N : \xi_k\neq \eta_k\}$.
In this case, we define $d_t(\xi,\eta) = \varepsilon(t,n)$.
The following observations are straightforward:
\begin{lem}\label{lem: properties d_t}
	\hspace{3em}
	\begin{enumerate}
		\item $d_t$ is an invariant (ultra-) metric on $\overleftarrow{G_*}$.
		\item $\varepsilon$-balls w.r.t.\ $d_t$ are cylinder sets.
		\item $M_{\varepsilon(t,n)} = [G:\Gamma_n] = sp^n$, where  $M_\varepsilon$ denotes the maximal cardinality of a $\varepsilon$-separated subset of $\overleftarrow{G_*}$ (w.r.t.\ the metric $d_t$).
		\item For all $n\in \N$ we have $0<\varepsilon(t,n)\leq 1$.
		Moreover, $\varepsilon(t,n)$ is decreasing in $n$ and we have $\varepsilon(t,n)\xrightarrow{n\to\infty} 0$ as well as $\frac{\varepsilon(t,n)}{\varepsilon(t,n+1)}\leq p$.
	\end{enumerate}
\end{lem}

\begin{lem}\label{lem: box dim d_t}
	$\mathrm{Dim}_{\mathrm{Box}}(\overleftarrow{G_*},d_t) =  \kl 1- t\cdot \frac{\log(p-2)}{\log p}\kr^{-1} =\kl (1-t) + t\cdot\frac{\log p}{\log p - \log(p-2)}\kr$.
\end{lem}
\begin{proof}
	Lemma~\ref{lem: properties d_t} (4) implies $\mathrm{Dim}_{\mathrm{Box}}(\overleftarrow{G_*},d_t) = \lim_{n\to\infty} \frac{\log M_{\varepsilon(t,n)}}{\log \varepsilon(t,n)}$ (see \cite[p.\ 41]{Falconer1990Fractal}).
	We therefore compute
	\begin{align*}
		\mathrm{Dim}_{\mathrm{Box}}(\overleftarrow{G_*},d_t) &= \lim_{n\to\infty} \frac{n\log p + \log s}{n \log p -  \#\{1\leq k\leq n: k\in N(t)\} \log(p-2)}\\
		&= \kl \lim_{n\to\infty} \frac{n\log p - \#\{1\leq k\leq n: k\in N(t)\} \log(p-2)}{n\log p} \kr^{-1} \\
		&\hspace{1cm}+ \underbrace{\lim_{n\to\infty} \frac{\log s}{n\log p - \#\{\ldots\}\log(p-2)}}_{=0}\\
		&= \kl 1-t\cdot \frac{\log(p-2)}{\log p}\kr^{-1}.
	\end{align*}
	This shows the first equality.
	The second equality is elementary.
\end{proof}

We now wish to construct a window $W := W_t \in \mathcal{W}_{\mathrm{reg}}(\overleftarrow{G_*})$ such that $d_W$ (defined as in Section~\ref{sec: ac}) is Lipschitz-equivalent to $d_t$, in particular yields the same box dimension.
The approach is very similar to \cite[Section 7.2]{CortezDrewloGomezJager2025Model}.\\
Choose a partition $\{A_n,B_n,C_n\}$ of $D_n\cap\Gamma_{n-1}$ such that for all $n\in\N$ we have $\# A_n \geq 1$ and $B_n = \{1_G\}$.
Moreover, we want that
$$\# C_1 = \begin{cases} s(p-2) &,\text{ if } 1\in N(t) \\ s &, \text{ otherwise} \end{cases} \text{ as well as }\# C_n = \begin{cases} p-2 &,\text{ if } n \in N(t) \\ 1 &, \text{ otherwise} \end{cases} \text{ for } n>1.$$
(By \eqref{eq: size D_n cap Gamma_{n-1}} such a choice is always possible.)
We now define
\begin{align*}
    &X_n = \prod_{k<n} C_k \times A_n \times \prod_{k>n}(D_k\cap\Gamma_{k-1}),\\
    &Y_n = \prod_{k<n} C_k \times B_n \times \prod_{k>n}(D_k\cap\Gamma_{k-1}), \text{ and}\\
    &Z_n = \prod_{k\leq n}C_k \times \prod_{k>n}(D_k\cap\Gamma_{k-1}).
\end{align*}
Lastly, let $W = \ol{\bigcup_{n\in\N} X_n}$.
The following observations are essentially \cite[Lem.\ 7.6, Lem.\ 7.7 and Cor.\ 7.10]{CortezDrewloGomezJager2025Model}.
However, since the authors present their construction (and hence the proofs) for the inverse limit representation $\overleftarrow{G}$ of the odometer, we reiterate the proofs for $\overleftarrow{G_*}$.
\begin{lem}
    The window is proper, generic and regular with $\mathrm{int}(W) = \bigcup_{n\in\N} X_n =: U$ and $\partial W = \bigcap_{n\in\N} Z_n$.
\end{lem}
\begin{proof}
    First, observe that with $V := \bigcup_{n\in\N} Y_n$ we have $\overleftarrow{G_*} = U\uplus V \uplus \bigcap_{n\in\N} Z_n$.
    We clearly have $U\tm \mathrm{int}(W)$, $V\tm \overleftarrow{G_*}\setminus W$ and hence $\partial W\tm \bigcap_{n\in\N} Z_n$.
    The first inclusion implies properness.
    Now let $\xi \in\bigcap_{n\in\N} Z_n$. To show that $\xi\in \partial W$, it suffices to find for every $n\in\N$ some $\eta_1 \in [\xi]_n \cap W$ and some $\eta_2\in [\xi]_n \setminus W$.
    For this, just pick $a_{n+1}\in A_{n+1}$ and arbitrary $\eta_k \in D_k\cap \Gamma_{k-1}$ for $k>n+1$.
    It is easy to verify that $\eta_1 = (\xi_1,\ldots, \xi_n, a_{n+1}, \eta_{n+2},\ldots)$ and $\eta_2 = (\xi_1,\ldots,\xi_n, 1_G, \eta_{n+2},\ldots)$ do the job.
    This shows $\partial W = \bigcap_{n\in\N} Z_n$ and hence also $\mathrm{int}(W) = U$.\\
    We can now compute
    $$ \nu(\partial W) = \lim_{n\to\infty} \nu(Z_n) = \lim_{n\to\infty} \prod_{k=1}^n \frac{\# C_k}{\# (D_k\cap \Gamma_{k-1})} \leq \lim_{n\to\infty} \kl\frac{p-2}{p}\kr^n = 0,$$
    so that $W$ is regular.
    If we let $\tau_*$ denote the canonical embedding of $G$ into $\overleftarrow{G_*}$, then by \eqref{eq: chara image of tau_*} that the elements of $\tau_*(G)$ are exactly the sequences $\xi$ which are eventually constant $1_G$.
    Together with the fact that $1_G\notin Z_n$ for every $n\in\N$, we obtain genericity of $W$.
\end{proof}

It remains to show the desired Lipschitz-equivalence.
\begin{lem}\label{lem: Lipschitz equiv}
	The pseudometric $d_W$ defined via $d_W(\xi,\eta) = \nu(W\xi^{-1}\symdiff W\eta^{-1})$ is Lipschitz-equivalent to $d_t$.
	In particular, $W$ is irredundant and $d_W$ is a metric.
\end{lem}
\begin{proof}
	Since both pseudometrics $d_t$ and $d_W$ are left-invariant, it suffices to show
	$$ \frac{1}{sp^2} \cdot d_t(\xi,1_{\overleftarrow{G_*}}) \leq d_W(\xi,1_{\overleftarrow{G_*}}) \leq p \cdot d_t(\xi,1_{\overleftarrow{G_*}})$$
	for all $\xi\in \overleftarrow{G_*}$.
	If $\xi = 1_{\overleftarrow{G_*}}$, there is nothing to prove.
	Therefore let $n = \min\{k\in \N: \xi_k \neq (1_{\overleftarrow{G_*}})_k = 1_G\}$, so that $d_t(\xi,1_{\overleftarrow{G_*}}) = \varepsilon(t,n)$.
	As $\xi_k =  1_G$ for all $k<n$, Lemma~\ref{lem: carry-over rule light version} implies that (right) multiplication with $\xi$ leaves all cylinders of level $k<n$ fixed.
	In particular, $\bigcup_{k<n} X_k\xi = \bigcup_{k<n} X_k$ and $\bigcup_{k<n} Y_k \xi = \bigcup_{k<n} Y_k$.
	Because $\overleftarrow{G_*} = \bigcup_{k<n} X_k \uplus\bigcup_{k<n} Y_k \uplus Z_{n-1}$, we deduce $W\xi^{-1} \symdiff W \tm Z_{n-1}$.
	Hence,
	\begin{align*}
		d_W(\xi,1_{\overleftarrow{G_*}}) &= \nu(W\xi^{-1}\symdiff W) \leq \nu(Z_{n-1})
		= \prod_{k=1}^{n-1} \frac{\#C_k}{\# (D_k\cap \Gamma_{k-1})}\\
		&= p^{-\#\{1\leq k \leq n-1\: :\: k\notin N(t)\}} \cdot \kl\frac{p}{p-2}\kr^{-\#\{1\leq k\leq n-1\: :\: k\in N(t)\}}\\
		&= \varepsilon(t,n-1)
		\leq p\cdot \varepsilon(t,n)
		= p\cdot d_t(\xi,1_{\overleftarrow{G_*}}).
	\end{align*}
	On the other hand, by construction of $Y_n$ and $B_n$ we have
	$$Y_n = \prod_{k<n} C_k \times \{1_G\} \times \prod_{k>n} (D_k\cap\Gamma_{k-1}) \tm \overleftarrow{G_*}\setminus W.$$
	Since $\xi_n \neq 1_G$ we either have $\xi_n\in A_n$ or $\xi_n\in C_n$.
	The first case implies $Y_n \xi \tm X_n$, so that $Y_n \tm W\xi\symdiff W$ and
	\begin{align*}
		d_W(\xi,1_{\overleftarrow{G_*}}) \geq \nu(Y_n) \geq \frac{1}{sp}\cdot \varepsilon(t,n-1) \geq \frac{1}{sp} \cdot \varepsilon(t,n) = \frac{1}{sp}\cdot d_t(\xi,1_{\overleftarrow{G_*}}).
	\end{align*}
	In the second case we can pick $h\in D_{n+1}\cap \Gamma_n$ such that $h\xi_{n+1} \in A_{n+1} \Gamma_{m+1}$.
	Now it is not hard to see that
 	$$ \prod_{k<n} C_k \times \{1_G\}\times \{h\} \times \prod_{k>n+1}(D_k\cap\Gamma_{k-1}) \tm W\xi^{-1}\symdiff W.$$
	Hence,
	\begin{align*}
		d_W(\xi,1_{\overleftarrow{G_*}}) \geq \frac{1}{sp^2}\cdot \varepsilon(t,n-1) \geq \frac{1}{sp^2} \cdot d_t(\xi,1_{\overleftarrow{G_*}}),
	\end{align*}
	finishing the proof of the Lipschitz equivalence.
	Since $d_t$ is positive definite, Lipschitz equivalence implies that $d_W$ is positive definite as well.
	Therefore, whenever $0= \nu(W\symdiff W\xi) = d_W(1_{\overleftarrow{G_*}},\xi^{-1})$ we have $\xi^{-1} = 1_{\overleftarrow{G_*}} = \xi$, showing irredundancy of $W$.
\end{proof}
Now we arrive at the main conclusion of this subsection.
Recall that the sequence $(\Gamma_n)_{n\in\N}$ (and hence the odometer $\overleftarrow{G_*}$) are fixed.
\begin{thm}
	For any $a\in [1,\infty)$ there exists a regular Toeplitz array $x\in\{0,1\}^G$ which is almost automorphic over $\overleftarrow{G_*}$ and such that $\mathrm{ac}(\ol{O_G(x)},G) = a$.
\end{thm}
\begin{proof}
	Assume first that $a \leq \frac{\log p}{\log p - \log(p-2)}$.
	Then, by Lemma~\ref{lem: box dim d_t} and Lemma~\ref{lem: Lipschitz equiv} there exists a window $W\in \mathcal{W}_{\mathrm{reg}}(\overleftarrow{G_*})$ such that $\mathrm{Dim}_{\mathrm{Box}}(\overleftarrow{G_*},d_W) = a$.
	Corollary~\ref{cor: ac is box dim of d_W} now yields that $\mathrm{ac}(\ol{O_G(x_W)},G) = a$ and Theorem~\ref{thm: 1-1 regular windows regular Toeplitz} shows that $x_W$ is a regular Toeplitz array which is almost automorphic over $\overleftarrow{G_*}$.
	Thus, we can realize any amorphic complexity between $1$ and $L(p):= \frac{\log p}{\log p - \log(p-2)}$.
	Now observe that for fixed $K\in \N$, the subsequence $\Gamma_n^\prime = \Gamma_{Kn}$ satisfies $[G:\Gamma_n^\prime] = sp^{Kn}$.
	Therefore, we can realize any amorphic complexity between $1$ and $L(p^K)$ with a regular Toeplitz array $x$ which is almost automorpic over the odometer $\overleftarrow{G^\prime_*}$ associated to $(\Gamma_n^\prime)_{n\in\N}$.
	However, recall that $\overleftarrow{G^\prime_*}$ and $\overleftarrow{G_*}$ are isomorphic as topological groups by \cite[Lem.\ 2]{CortezPetite2008Odometers}.
	Letting $K\to \infty$ finishes the proof.
\end{proof}

In order to prove Theorem~\ref{thm: realize ac Toeplitz intro}, it remains to show that all finitely generated, torsion-free, nilpotent groups $G$ are amenable and there exist $s\geq 1$, $p\geq 3$ and a decreasing sequence $(\Gamma_n)_{n\in\N}$ of finite index normal subgroups with $\bigcap_{n\in\N} \Gamma_n = \{1_G\}$ and $[G:\Gamma_n] = sp^n$.
Note that amenability already follows from nilpotency.
Regarding the second condition, an even stronger result holds which can be proved using known results from the theory of (analytic) pro-$p$ groups.

\begin{prop}
	Let $G$ be a finitely generated, torsion-free, nilpotent group.
	Then, for every odd prime $q$ there exist $s,h\geq 1$ and a sequence $(\Gamma_n)_{n\in\N}$ of finite index normal subgroups such that $\bigcap_{n\in\N}\Gamma_n = \{1_G\}$ and $[G:\Gamma_n] = sq^{hn}$ for all $n\in\N$.
\end{prop}
\begin{proof}
	By \cite[Prop.\ 2(iv) in Chapter "Window 6; Soluble Groups"]{LubotzkySegal2003SubgroupGrowth}, the group $G$ has a central series of finite length $h$ such that every factor is isomorphic to $\Z$.
	(The integer $h$ is an invariant of $G$ called the \textbf{Hirsch length}.)
	Let $q$ be an odd prime.
	By (iii) of the same proposition, $G$ embeds densely into its pro-$q$ completion $\widehat{G}_q$ and by (v) $\widehat{G}_q$ is a pro-$q$ group of (finite) rank and dimension $h$.
	Now, \cite[Cor.\ 8.34]{DixonDusautoyMannSegal2003Analytic} yields that $\widehat{G}_q$ contains an open, normal, uniform pro-$q$ subgroup $U$ of finite index.
	It follows for each $n$ that the subgroup $U_n := \langle u^{q^n} : u\in U\rangle$ is open and normal in $\widehat{G}_q$.
	Moreover, by definition of $\mathrm{dim}(\widehat{G}_q)$ (see \cite[Def.\ 4.7]{DixonDusautoyMannSegal2003Analytic}), there exists a subset $\{u_1,\ldots,u_h\}$ of $U$ which topologically generates $U$.
	In \cite[Chapter 4.3]{DixonDusautoyMannSegal2003Analytic} it is shown that there exists an additive structure on $U$ which turns $U$ into a free $\Z_q$-module with basis $\{u_1,\ldots,u_h\}$.
	(Here, $\Z_q$ denotes the $q$-adic integers.)
	Furthermore, we have $U_n = q^n U$ and the multiplicative cosets of $U_n$ coincide with the additive cosets.
	From this it follows that $[U:U_n] = q^{hn}$ for all $n\in\N$.
	If we now let $\Gamma_n = G\cap U_n$ (where we identify $G$ as a subgroup of $\widehat{G}_q$), then by the second isomorphism theorem $[G: \Gamma_n] = [GU_n :U_n]$.
	Since $U_n$ is open and $G$ is dense, we obtain $GU_n = \widehat{G}_q$, so that with $s = [\widehat{G}_q : U]$ one indeed has
	$$ [G: \Gamma_n] = [\widehat{G}_q: U_n] = s [U:U_n] = sq^{hn}.$$
	Lastly, clearly $\bigcap_{n\in\N} U_n$ contains only the identity of $\widehat{G}_q$, so that $\bigcap_{n\in\N} \Gamma_n = \{1_G\}$, concluding the proof.
\end{proof}

\section{A strengthening of Krieger's Theorem}\label{sec: krieger}
Using techiques similar to \cite{LackaStraszak2018QuasiUniform} we will show in this section Theorem~\ref{thm: realize entropy intro}, which is a strengthening of Krieger's Theorem \cite[Thm.\ 1.1]{Krieger2007Toeplitz} in the case of 0-1-sequences.
Before we can give the construction, we shall first generalize results from \cite{LackaStraszak2018QuasiUniform} regarding path connectedness.

\subsection{Path connectedness}\label{sec: path connectedness}
Fix a decreasing sequence $\mathbf{\Gamma} = (\Gamma_n)_{n\in\N}$ of finite index normal subgroups of $G$ such that $\bigcap_{n\in\N} \Gamma_n = \{1_G\}$.
In \cite[Thm.\ 56]{LackaStraszak2018QuasiUniform} it is shown that $\mathrm{Rel}(\mathbf{\Gamma})$ is path-connected w.r.t.\ $D_{\mathrm{Weyl}}$.
In order to prove this, the authors also contruct a path in $\mathrm{Rel}(\mathbf{\Gamma})$ between the constant 0-array and the constant 1-array, so that every point on that path is a regular Toeplitz array relative to $\mathbf{\Gamma}$ (see \cite[Lem.\ 54]{LackaStraszak2018QuasiUniform}).
We follow the same idea to show the following two statements:
Let $(H,\tau)$ be a metrizable group compactification of $G$.

\begin{prop}\label{prop: path between empty and full}
	There exists a continuous (w.r.t.\ $\ol{D}$) path in $\mathfrak{B}_{\mathrm{prop}}\cap\mathfrak{B}_{\mathrm{gen}}(H)\cap \mathfrak{B}_{\mathrm{reg}}(H)$ between $\emptyset$ and $H$.
\end{prop}

\begin{thm}\label{thm: generic windows path connected}
	The spaces $\mathfrak{B}(H)$, $\mathfrak{B}_{\mathrm{gen}}(H)$, $\mathfrak{B}_{\mathrm{reg}}(H)$, $\mathfrak{B}_{\mathrm{closed}}(H)$ and all possible intersections thereof are path-connected w.r.t.\ $\ol{D}$.
\end{thm}

Observe that Theorem~\ref{thm: properties Phi} and Proposition~\ref{prop: relative Toeplitz generic windows} show that the aforementioned \cite[Lem.\ 54, Thm.\ 56]{LackaStraszak2018QuasiUniform} are indeed a special case of Proposition~\ref{prop: path between empty and full} and Theorem~\ref{thm: generic windows path connected} (taking $H= \overleftarrow{G}$).\\
In order to prove these two claims, we need a specific refining sequence of measurable partitions.
The construction is fairly standard, but since we aim for a specific collection of properties of these partitions, we still give the details.
In what follows, we freely identify $g\in G$ with its embedding $\tau(g)\in H$, whenever there is no ambiguity between the two.
We call a subset of the positive real line \textbf{co-countable} if its complement is countable.
\begin{lem}\label{lem: good boundary of ball}
	Let $\xi\in H$.
	\begin{enumerate}
		\item The set $\{\delta>0 : \nu(\partial B_\delta(\xi))=0\}$ is co-countable.
		\item The set $\{\delta>0 : \partial B_\delta(\xi)\cap \tau(G) = \emptyset\}$ is co-countable.
	\end{enumerate}
\end{lem}
\begin{proof}
	(1) Observe that $(0,\infty) \setminus \{\delta>0: \nu(\partial B_\delta(\xi))=0\} \tm \{\delta>0: r\mapsto \nu(B_r(\xi)) \text{ is not continuous at } \delta\}$.
	However, the latter set is countable, since the map $r\mapsto \nu(B_r(\xi))$ is increasing.\\
	(2) Let $\Delta := \{\delta>0 : \partial B_\delta(\xi)\cap \tau(G) \neq \emptyset\}$.
	For each $\delta \in \Delta$ choose $g_\delta \in G$ such that $\tau(g_\delta) \in \partial B_\delta(\xi)$.
	Now note that, whenever $\delta \neq \delta^\prime$, then $\partial B_\delta(\xi) \cap \partial B_{\delta^\prime}(\xi) = \emptyset$.
	Thus, by injectivity of $\tau$, we see that the mapping $\Delta \to G,\: \delta\mapsto g_\delta$ is injective.
	As we assumed $G$ to be countable, this concludes the proof.
\end{proof}

\begin{lem}\label{lem: good partition}
	There exists a sequence $(\mathcal{P}_n)_{n\in\N}$ of finite families $\mathcal{P}_n = \{P_{n,i} : i\in \{1,\ldots,m_n\}\}$ of measurable subsets of $H$, such that
	\begin{enumerate}
		\item $P_{n,i}\cap P_{n,j} = \emptyset$ whenever $i\neq j$ and $\nu\kl \bigcup_{i=1}^{m_n} P_{n,i} \kr =1$.
		In particular, each $\mathcal{P}_n$ is a measurable partition of $H$.
		\item Each set $P_{n,i}$ is open.
		\item $\nu(\partial P_{n,i}) = 0$ for all $n\in \N$, $i\in \{1,\ldots,m_n\}$.
		\item $\partial P_{n,i} \cap \tau(G) = \emptyset$ for all $n$ and $i$.
		\item $\tau(G) \tm \bigcup_{i=1}^{m_n} P_{n,i}$ for all $n$.
		\item $\mathcal{P}_{n+1}$ is a refinement of $\mathcal{P}_n$.
		\item $\max\{\nu(P): P\in \mathcal{P}_n\} \xrightarrow{n\to\infty} 0$.
		\item $\max\{ \mathrm{diam}(P) : P \in \mathcal{P}_n\} \xrightarrow{n\to\infty} 0$.
	\end{enumerate}
\end{lem}
\begin{proof}
	By Lemma~\ref{lem: good boundary of ball}, we can find a decreasing sequence $(\varepsilon_n)_{n\in\N}$ of positive real numbers with $\varepsilon_n \to 0$ and such that for all $n\in\N$ we have $\nu(\partial B_{\varepsilon_n}(1_H)) = 0$ and $\partial B_{\varepsilon_n}(1_H) \cap \tau(G) = \emptyset$.
	Since the measure $\nu$ and our metric on $H$ are both invartiant and $\tau$ is a group homomorphism, this implies $\nu(\partial B_{\varepsilon_n}(\xi)) = 0$ for all $\xi \in H$ as well as $\partial B_{\varepsilon_n}(\tau(h)) \cap \tau(G) = \emptyset$ for all $h\in G$.\\
	Let $n\in \N$ be arbitrary.
	By denseness of $\tau(G)$ in $H$, we can find $h_{n,1},\ldots,h_{n,m_n^\prime} \in G$ such that $\{ B_{\varepsilon_n}(h_{1,i}) : i = 1,\ldots,m_n^\prime\}$ covers $H$.
	We inductively define
	$$ P^\prime_{n,1} = B_{\varepsilon_n}(h_{n,1}) \text{ and } P^\prime_{n,i} = B_{\varepsilon_n}(h_{n,i}) \setminus \bigcup_{j=1}^{i-1} \ol{P^\prime_{n,j}} \text{ for } i>1.$$
	It follows by induction that $\partial P^\prime_{n,i} \tm \bigcup_{j=1}^i \partial B_{\varepsilon_n}(h_{n,j})$.
	Hence,
	$$H\setminus \bigcup_{i=1}^{m_n^\prime} P^\prime_{n,i} \tm \bigcup_{i=1}^{m_n^\prime} \partial P_{n,i}^\prime \tm \bigcup_{i=1}^{m_n^\prime} \partial B_{\varepsilon_n}(h_{n,i}).$$
	This readily yields properties (1), (3), (4) and (5) for the partition $\mathcal{P}^\prime_n:= \{P^\prime_{n,i}: i=1,\ldots,m_n^\prime\}$, while (2) is obvious.
	Now it remains to define $\mathcal{P}_1 = \mathcal{P}^\prime_1$ and $\mathcal{P}_{n+1} = \mathcal{P}_n \vee \mathcal{P}_{n+1}^\prime$.
	The remaining properties (6) - (8) are easily verified.
\end{proof}
\begin{rem}
	If $G$ is residually finite, $H =\overleftarrow{G_*} = \prod_{n\in\N} (D_n\cap \Gamma_{n-1})$ is a $G$-odometer and $\tau_*: G\to H$ is the natural embedding (c.f.\ Section~\ref{sec: odometers}), then
	\begin{align*}
		\mathcal{P}_n &= \{P\tm H: P\text{ is a cylinder set of level } n\} \\
			&= \ml[\xi_1,\ldots,\xi_n] : (\xi_1,\ldots,\xi_n)\in \prod_{k=1}^n (D_n\cap \Gamma_{n-1})\mr
	\end{align*}
	gives us a sequence of partitions that satisfies the properties of Lemma~\ref{lem: good partition}.
	Recall that by definition of the embedding $\tau_*$ we have
	\begin{equation}\label{eq: decoding of element from G odometer}
	\xi\in \tau_*(G) \text{ if and only if } \xi = (\xi_n)_{n\in\N} \text{ is eventually constant } 1_G.
	\end{equation}
	For the situation of a general group compactification $(H,\tau)$ it is now our aim to re-enumerate the partitions given by Lemma~\ref{lem: good partition} so that we have a characterization of the elements in $\tau(G)$ analogous to \eqref{eq: decoding of element from G odometer}.
\end{rem}

Of course, by discarding empty sets if necessary we can assume that all $P\in \mathcal{P}_n$ are non-empty.
Take some arbitrary increasing sequence $(G_n^\prime)_{n\in\N}$ of finite subsets of $G$ such that $1_G\in G_1^\prime$ and  $G = \bigcup_{n\in\N} G_n^\prime$.
By making $\varepsilon_n$ in above proof smaller, we can assume that $\mathcal{P}_n$ separates the points in $\tau(G_n^\prime)$.
We re-enumerate the elements of $\mathcal{P}_n$ as follows: \\
Write $\mathcal{P}_1 = \{P(1),\ldots,P(m_1)\}$, where $P(1)$ is the unique $P\in \mathcal{P}_1$ such that $1_H \in P$.
This means that we can write $\mathcal{P}_1 = \{P(l_1) : l_1 \in Z_1\}$ with $Z_1 = \{1,\ldots,m_1\} \tm \N^1$.
Also choose a set $G_1 = \{g(l_1): l_1\in Z_1\}\tm G$ such that $\tau(g(l_1))\in P(l_1)$ and $G_1^\prime \tm G_1$.
(This can be done since we assumed that $\mathcal{P}_1$ separates the points in $\tau(G_1^\prime)$.)\\
Assume for an induction that we have enumerated $\mathcal{P}_n = \{P(l_1,\ldots,l_n) : (l_1,\ldots,l_n) \in Z_n\}$ with some (finite) $Z_n \tm \N^n$ and chosen $G_n := \{g(l_1,\ldots,l_n) : (l_1,\ldots,l_n)\in Z_n\}\tm G$ such that $\tau(g(l_1,\ldots,l_n))\in P(l_1,\ldots,l_n)$ and $G_n^\prime \tm G_n$.
We see by Lemma~\ref{lem: good partition} (6), that for every $P\in \mathcal{P}_{n+1}$ there exists a unique $(l_1,\ldots,l_n)\in Z_n$ such that $P\tm P(l_1,\ldots,l_n)$.
We enumerate
$$ \{ P\in \mathcal{P}_{n+1} : P\tm P(l_1,\ldots,l_n)\} = \{ P(l_1,\ldots,l_n,l_{n+1}) : l_{n+1} = 1,\ldots, m(l_1,\ldots,l_n)\}$$
so that $P(l_1,\ldots,l_n,1)$ is the unique $P$ with $\tau(g(l_1,\ldots,l_n)) \in P$.
This gives a enumeration/ encoding $\mathcal{P}_{n+1} = \{P(l_1,\ldots,l_{n+1}) : (l_1,\ldots,l_{n+1})\in Z_{n+1}\}$,
where
$$Z_{n+1} = \{(l_1,\ldots,l_n,l_{n+1}) : (l_1,\ldots,l_n)\in Z_n \text{ and } l_{n+1} \in \{1,\ldots,m(l_1,\ldots,l_n)\}\}.$$
Lastly, choose $G_{n+1} := \{g(l_1,\ldots,l_{n+1}): (l_1,\ldots,l_{n+1}) \in Z_{n+1} \}\tm G$ such that
\begin{align}\label{eq: def G_{n+1}}
	&\tau(g(l_1,\ldots,l_{n+1}))\in P(l_1,\ldots,l_{n+1}), \nonumber \\
	&g(l_1,\ldots,l_n,1) = g(l_1,\ldots,l_n) \text{ and} \\
	&G_{n+1}^\prime \tm G_{n+1}. \nonumber
\end{align}
(The last property can be ensured since $\mathcal{P}_{n+1}$ separates the points in $\tau(G_{n+1}^\prime)$.)
This finishes the inductive construction.\\
We see that every $\xi \in H^\prime :=\bigcap_{n\in\N} \bigcup_{P\in \mathcal{P}_n} P$ has a unique encoding $(l_n)_{n\in\N}$, so that for every $n\in\N$ the vector $(l_1,\ldots,l_n)$ is the unique element of $Z_n$ such that $\xi \in P(l_1,\ldots,l_n)$.
(Uniqueness of this encoding follows from Lemma~\ref{lem: good partition} (8)).
Also note that $\tau(G) \tm H^\prime$ by (5).
With the sets $G_n$ defined as above, it follows from \eqref{eq: def G_{n+1}} that $\xi\in\tau(G_n)$ if and only if the encoding $(l_m)_{m\in\N}$ of $\xi$ satisfies $l_m = 1$ for all $m>n$.
Moreover, $G_n\tm G_{n+1}$ and $\bigcup_{n\in\N} G_n = G$.
This yields the desired analogue to \eqref{eq: decoding of element from G odometer}.
\begin{lem}\label{lem: decoding of element from G}
	Let $\xi \in H^\prime$.
	We have $\xi \in \tau(G)$ if and only if the encoding $(l_n)_{n\in\N}$ of $\xi$ is eventually constant 1.
\end{lem}

Now we can turn to the proofs of Proposition~\ref{prop: path between empty and full} and Theorem~\ref{thm: generic windows path connected}, which follow along the same lines as \cite{LackaStraszak2018QuasiUniform}.
\begin{proof}[Proof of Proposition~\ref{prop: path between empty and full}]
	Let $t\in [0,1]$.
 	Define
 	\[k_1 :=k_1(t) = \max \ml i : \nu \kl \bigcup_{j=1}^{i-1} P_j \kr \leq t \mr \text{ and } A_1(t) =\bigcup_{j=1}^{k_1-1} P_j.\]
	Now assume inductively we have defined $k_1,\ldots,k_n$ and $A_n(t)$.
	Let $k_{n+1} := k_{n+1}(t) = \max\ml i: \nu \kl A_n(t) \cup \bigcup_{j=1}^{i-1} P(k_1,\ldots,k_n,j)\kr\leq t\mr$ as well as
	$$A_{n+1}(t) = A_n(t) \cup\bigcup_{j=1}^{k_{n+1}-1} P(k_1,\ldots,k_n,j).$$
	We let $A(t) = \ol{\bigcup_{n\in\N} A_n(t)}$ and claim that $t\mapsto A(t)$ is the desired path.\\
	First, it is clear that $A_n(0) =\emptyset$ for all $n\in\N$, so that $A(0) = \emptyset$.
	On the other hand, we have $A_1(1) = \bigcup_{P\in \mathcal{P}_1} P$ and $A_n(1) = \emptyset$ for $n>1$.
	From Lemma~\ref{lem: good partition} (5), one obtains $A(1) = \ol{A_1(1)} = H$.
	In order to see that $A(t)$ is proper, generic and regular, note that each $A_n(t)$ is an open set, so that $\bigcup_{n\in\N} A_n(t) \tm \mathrm{int}(A(t))$.
	This yields properness.
	Furthermore,
	\begin{equation}\label{eq: boundary of A(t)}
	 \partial A(t) \tm \bigcup_{n\in\N} \bigcup_{P\in \mathcal{P}_n} \partial P \cup \bigcap_{n\in\N} P(k_1,\ldots,k_n).
	\end{equation}
	Regularity thus follows from Lemma~\ref{lem: good partition} (3) and (7).
	For genericity, first assume that $\nu(A_n(t)) = t$ for some $n\in\N$.
	In this case, we have $A(t) = \ol{A_n(t)}$ and $A_n(t)$ is a (finite) union of elements from $\mathcal{P}_m$ with $m\leq n$.
	Thus, genericity follows from Lemma~\ref{lem: good partition} (4).\\
	In the other case, where $\nu(A_n(t)) < t$ for all $n\in\N$, then by (7) we must have $k_n\geq 2$ for infinitely many $n$.
	Therefore, by Lemma~\ref{lem: decoding of element from G}, the (singleton) $\bigcap_{n\in\N} P(k_1,\ldots,k_n)$ does not intersect $\tau(G)$.
	Together with \eqref{eq: boundary of A(t)} and (4), this implies genericity in this case.\\
	It remains to show that $t\mapsto A(t)$ is continuous.
	Let $\varepsilon>0$ and choose $N\in\N$ such that $\max\{\nu(P) : P\in \mathcal{P}_N\} < \varepsilon$.
	Let $\delta = \min\{\nu(P) : P\in \mathcal{P}_N\} >0$.
	Note that this also yields $\min\{ \nu(P) : P\in \mathcal{P}_n\} \geq \delta$ for all $n\leq N$, since $\mathcal{P}_N$ refines $\mathcal{P}_n$.
	This implies that whenever $\abs{t-s} < \delta$, then $k_n := k_n(t) = k_n(s)$ for all $n\leq N$.
	From that, we readily obtain $A(t)\symdiff A(s) \tm \ol{P(k_1,\ldots,k_N)}$, hence also $\ol{A(t)\symdiff A(s)}\tm \ol{P(k_1,\ldots,k_N)}$.
	We therefore conclude, whenever $\abs{t-s} < \delta$, then
	$$ \ol{D}(A(t),A(s)) \leq \nu(\ol{P(k_1,\ldots,k_N)}) = \nu(P(k_1,\ldots,k_N)) < \varepsilon. \qedhere$$
\end{proof}

\begin{proof}[Proof of Theorem~\ref{thm: generic windows path connected}]
	Let $W, W^\prime \in \mathfrak{B}(H)$.
	We let
	$$W(t) = (W^\prime\cap A(t)) \cup (W \setminus \mathrm{int}(A(t))) = (W^\prime\cap A(t))\cup(W \setminus A(t)) \cup (W\cap \partial A(t)).$$
	Observe that
	$$ W(t) \symdiff W(s) \tm (A(t) \symdiff A(s)) \cup (\partial A(t) \symdiff \partial A(s)) \tm (A(t) \symdiff A(s)) \cup \partial A(t) \cup \partial A(s).$$
	Hence, $\ol{D}(W(t),W(s)) \leq \ol{D}(A(t),A(s)) + \nu(\ol{\partial A(t)\cup \partial A(s)}) = \ol{D}(A(t),A(s))$, since $A(t),A(s)$ are regular.
	Together with $W(0) = W$ and $W(1) = W^\prime$, we obtain that $t\mapsto W(t)$ is indeed a continuous path from $W$ to $W^\prime$.
	From the definition, it is also clear that each $W(t)$ is a closed set, whenever both $W$ and $W^\prime$ are closed.
	Furthermore,
	$$ \partial W(t) \tm \partial W^\prime \cup \partial A(t) \cup \partial W,$$
	so that we have $W(t) \in \mathfrak{B}_{\alpha}(H)$ whenever $W,W^\prime \in \mathfrak{B}_{\alpha}(H)$ for $\alpha\in\{\mathrm{gen},\mathrm{closed},\mathrm{reg}\}$.
	This concludes the proof.
\end{proof}

\subsection{The construction}
%For this section, we let $G$ be any countably infinite, residually finite, amenable group and let $(H,\tau)$ be a compact, metrizable group compactification of $G$ with Haar measure $\nu$.
Note that with Theorem~\ref{thm: properties Phi} and Theorem~\ref{thm: entropy cts wrt Weyl} we obtain
\begin{cor}\label{cor: entropy continuous}
	The mapping $\mathfrak{B}(H) \to [0,\infty),\: W \mapsto \htop(\ol{O_G(x_W)},G)$ is continuous w.r.t.\ $\ol{D}$.
\end{cor}
In \cite{LackaStraszak2018QuasiUniform}, the authors gave an alternative proof of Krieger's Theorem using the aforementioned continuity of entropy together with certain results on path connectedness.
We will employ the same strategy.
In particular, similar to the proof of \cite[Thm.\ 57]{LackaStraszak2018QuasiUniform}, we only aim to construct windows $W\in \mathcal{W}(H)$ such that $x_W$ has entropy arbitrarily close to $\log 2$.
The realisation of any intermediate value will follow (with some extra care to ensure properness and irredundancy) from Corollary~\ref{cor: entropy continuous} and path-connectedness result Theorem~\ref{thm: generic windows path connected}.\\
Let $\gamma\in [0,1)$ be arbitrary.
We shall now give the construction for a window $W_\gamma \in \mathcal{W}(H)$ such that $\htop(\ol{O_G(x_{W_\gamma})},G) \geq \gamma \cdot\log 2$.
For this we need to have a large set of "good radii"
\[ \mathcal{E} = \{\varepsilon >0: \nu(\partial B_{\varepsilon}(\tau(g))) = 0 \text{ and } \partial B_{\varepsilon}(\tau(g)) \cap \tau(G) = \emptyset \text{ for all } g\in G\},\]
which is a co-countable subset of the positive real line by Lemma~\ref{lem: good boundary of ball} and invariance of the metric and measure on $H$.
In particular, this yields $\inf \mathcal{E} = 0$.
The following is an elementary -- yet important -- observation.
\begin{lem}\label{lem: strictly decreasing balls}
	For any $\varepsilon \in \mathcal{E}$ there exists $\delta \in \mathcal{E}$ with $\delta < \varepsilon$ such that $B_{\delta}(\xi)$ is a proper subset of $B_{\varepsilon}(\xi)$ for any $\xi\in H$.
\end{lem}
\begin{proof}
	Firstly observe that by invariance of the metric it suffices to prove the assertion with some fixed $\xi\in H$.
	Suppose for a contradiction that there exists $\varepsilon \in \mathcal{E}$ such that for all $\delta< \varepsilon$ we have $B_{\delta}(\xi)= B_{\varepsilon}(\xi)$.
	Since $\inf \mathcal{E} = 0$, this implies $\{\xi\} = \bigcap_{\delta< \varepsilon} B_{\delta}(\xi) = B_{\varepsilon}(\xi)$.
	Hence, $\xi$ is an isolated point of $H$.
	However, as $H$ is a topological group, every point in $H$ is isolated so that $H$ is discrete.
	This contradicts the fact that $H$ is compact and infinite.
\end{proof}
Now we can begin with the inductive construction of the window $W_\gamma$.
In what follows we will alternatively write $e$ for the neutral element of $G$ in order to alliviate some of the overloading on the symbol "$1$".
Since $G$ is residually finite, we fix some decreasing sequence $(\Gamma_n)_{n\in\N}$ of finite index, normal subgroups of $G$ such that $\bigcap_{n\in\N} \Gamma_n = \{e\}$ and let $(F_n)_{n\in\N}$ be a sequence of fundamental domains given by Lemma~\ref{lem: good fundamental domains}.
\begin{construction}
	For step 0, we let $F_{k_0} = \{e\}$, $g_0^* = e$ and $\varepsilon_0$ such that $B_{\varepsilon_0}(g_0^*) = H$.\\
	Now for step 1, choose $k_1 \in \N$ large enough such that
	\begin{itemize}
		\item $r_1 := \floor{(1-\gamma)\# F_{k_1}} \geq 4.$
		\item We can find $G_1 \tm F_{k_1}$ such that $e\in G_1$ and $\# G_1 = r_1-1$.
	\end{itemize}
	Define a $G_1$-word $\mathfrak{W}_1$ to be $1$ on $e$ and $0$ everywhere else.
	Now choose $\varepsilon_1 \in \mathcal{E}$ such that the sets $B_{\varepsilon_1}(f)$ ($f\in F_{k_1}$) are pairwise disjoint and there exists furthermore some $g_1^* \in G$ such that $B_{\varepsilon_1}(g_1^*)$ is disjoint from $B_{\varepsilon_1}(F_{k_1})$.
	We now let
	\begin{align*}
		U_1 &= \bigcup \{B_{\varepsilon_1}(\tau(f)) : f\in G_1 \text{ and } \mathfrak{W}_1(f) = 1\}\\
		V_1 &= \bigcup \{B_{\varepsilon_1}(\tau(f)) : f\in G_1 \text{ and } \mathfrak{W}_1(f) = 0\}.
	\end{align*}
	We need to define some more subsets of $H$ and $G$:
	Let $H_1 = U_1 \cup V_1$ and $H_1^* = H_1 \cup B_{\varepsilon_1}(\tau(g_1^*))$.
	We also define $T_1 = \tau^{-1}(H_1)$, $T_1^* = \tau^{-1}(H_1^*)$, $\widehat{T}_1 = \tau^{-1}(B_{\varepsilon_1}(1_H))$ and $S_1 = F_{k_1} \setminus T_1$.
	Observe that by $\varepsilon_1\in \mathcal{E}$, we have $\nu(\partial H_1) = 0$ and $\partial H_1 \cap \tau(G) =\emptyset$ and analogously for $H_1^*$.
%	\begin{rem}
%		Comparing this construction to the one in the preceding Subsection~\ref{sec: krieger/ac}, the sets $X_n$ and $Y_n$ serve the same purpose in both.
%		The set $X_n$ (resp.\ $Y_n$) denotes the points that are put into $\mathrm{int}(W)$ (resp.\ $H\setminus W$) after the $n$-th step.
%		Thus, $T_n$ describes the points in $G$, on which the sequence $x_W$ is defined during the $n$-th step.
%	\end{rem}
	\begin{namedtheorem}[Claim]
		%\hspace{1em}
		\begin{enumerate}
			\item $\#(F_{k_1} h\cap T_1) \leq \#(F_{k_1} h \cap T_1^*)\leq r_1$ for all $h\in G$.
			\item $\# S_1 \geq \gamma \cdot\#F_{k_1}$.
			\item $B_{\varepsilon_0}(g_0^*) \setminus H_1 = B_{\varepsilon_0}(g_0^*) \setminus (B_{\varepsilon_1}(g_0^*) \cup V_1)$.
			\item $F_{k_0} \tm T_1$.
			\item $B_{\varepsilon_1}(g_1^*)\cap H_1 = \emptyset$.
			\item $S_1 \cdot \widehat{T}_1 \tm G \setminus T_1^*$.
		\end{enumerate}
	\end{namedtheorem}
	%need manual proof because otherwise the \qed is a triangle
	\textit{Proof.}
		(1) The first inequality is trivial.
		For the second inequality, let $f\in F_{k_1}$ such that $fh \in T_1^*$.
		Then either $\tau(fh)\in H_1$ or $\tau(fh)\in B_{\varepsilon_1}(g_1^*)$.
		In the latter case we see that for any $f^\prime\in F_{k_1}$ with $\tau(f^\prime h)\in B_{\varepsilon_1}(g_1^*)$ we have $g_1^* \in B_{\varepsilon_1}(\tau(fh))\cap B_{\varepsilon_1}(\tau(f^\prime h))$, so that $g_1^*\tau(h)^{-1} \in B_{\varepsilon_1}(f)\cap B_{\varepsilon_1}(f^\prime)$.
		The choice of $\varepsilon_1$ yields $f^\prime = f$.
		Similarly, in the first case we have $\tau(fh) \in B_{\varepsilon_1}(g)$ for some $g\in G_1$ and if $\tau(f^\prime h) \in B_{\varepsilon_1}(g)$, then $f^\prime = f$.
		This yields $\#(F_{k_1}h \cap T_1^*) \leq \# G_1 + 1 = r_1$ as desired.\\
		(2) Taking $h = e$ in (1) gives us $\# S_1 \geq \#F_{k_1}- r_1 \geq \#F_{k_1}-(1-\gamma)\# F_{k_1} = \gamma\cdot\#F_{k_1}$.\\
		(3) is clear from the convention on $\varepsilon_0$ and $g_0^*$ as well as the fact that $U_1 = B_{\varepsilon_1}(\tau(e)) = B_{\varepsilon_1}(\tau(g_0^*))$.\\
		(4) is trivial and (5) is a consequence of the choice of $\varepsilon_1$.\\
		(6) If $s \in S_1$ and $t\in \widehat{T}_1$, then due to the choice of $\varepsilon_1$
		$$\tau(st)\in B_{\varepsilon_1}(s) \tm H\setminus \bigcup_{g\in (F_{k_1}\setminus\{s\}) \cup \{g_1^*\}} B_{\varepsilon_1}(g) \tm H\setminus H_1^*,$$
		so that $st \in G\setminus T_1^*$, as desired.
		\hfill
		$\square$\\
	Assume inductvely that we have constructed $k_n,r_n\in\N$, $\varepsilon_n\in \mathcal{E}$, $g_n^* \in G\setminus F_{k_n}$, as well as subsets $U_n$, $V_n$, $H_n$, $H_n^*$ of $H$ and subsets $T_n$, $T_n^*$, $\widehat{T}_n$ and $S_n$ of $G$ such that the following properties hold:
	\begin{enumerate}[leftmargin=5em]
		\item[(1-(n))] $\#(F_{k_n} h \cap T_n)\leq \#(F_{k_n}h \cap T_n^*) \leq r_n$ for all $h\in G$.
		\item[(2-(n))] $\# S_n \geq \gamma\cdot \#F_{k_n}$.
		\item[(3-(n))] $B_{\varepsilon_{n-1}}(g_{n-1}^*) \setminus \bigcup_{l=1}^n H_l = B_{\varepsilon_{n-1}}(g_{n-1}^*) \setminus (B_{\varepsilon_n}(g_{n-1}^*)\cup V_n)$.
		\item[(4-(n))] $F_{k_{n-1}}\tm \bigcup_{l=1}^n T_n$.
		\item[(5-(n))] $B_{\varepsilon_n}(g_n^*)\cap \bigcup_{l=1}^n H_l = \emptyset$.
		\item[(6-(n))] $S_n \cdot \widehat{T}_n \tm G \setminus \kl \bigcup_{l=1}^{n-1} T_l \cup T_n^*\kr$.
	\end{enumerate}
	Pick pairwise disjoint subsets $S_{n,j}\tm G$ ($j=1,\ldots,2^{\# S_n}$) such that each $S_{n,j}$ is of the form $S_{n,j}= S_n \cdot t_{n,j}$ with $t_{n,j}\in \widehat{T}_n$ and $S_{n,1} = S_n$.
	Observe that by (5-(n)) and (6-(n)) the union $\bigcup_{j=1}^{2^{\# S_n}} S_{n,j} \uplus \{g_n^*\}\uplus \bigcup_{l=1}^n T_l$ is indeed disjoint.
	We claim that one can find $g_{n,1},\ldots,g_{n,m_n} \in G \setminus \kl \bigcup_{j=1}^{2^{\# S_n}} S_{n,j} \uplus \{g_n^*\}\uplus \bigcup_{l=1}^n T_l\kr$ such that
	\begin{equation}\label{eq: covering property}
		B_{\varepsilon_{n-1}}(g_{n-1}^*) \setminus \bigcup_{l=1}^n H_l \tm \bigcup_{i=1}^{m_n} B_{\varepsilon_n}(g_{n,i}).
	\end{equation}
	To see that this claim is true, note that $B_{\varepsilon_{n-1}}(g_{n-1}^*) \setminus \bigcup_{l=1}^n H_l$ is totally bounded (as a subset of the totally bounded space $H$).
	Hence can cover it with balls of the form $B_{\varepsilon_n/2}(\xi_{n,i})$ (with $\xi_{n,i} \in B_{\varepsilon_{n-1}}(g_{n-1}^*) \setminus \bigcup_{l=1}^n H_l$ and $i = 1,\ldots,m_n$).
	Moreover, any set of the form
	\begin{equation}\label{eq: intersec difference and ball}
	\kl B_{\varepsilon_{n-1}}(g_{n-1}^*) \setminus \ol{\bigcup_{l=1}^n H_l}\kr \cap B_{\varepsilon_n/2}(\xi_{n,i}) \quad (i=1,\ldots,m_n)
	\end{equation}
	is non-empty and open, hence contains infinitely many $\tau(g)$ with $g\in G$.
	In particular, we can choose $g_{n,i} \notin \bigcup_{j=1}^{2^{\# S_n}} S_{n,j} \uplus \{g_n^*\}$ such that $\tau(g_{n,i})$ lies in the set described in \eqref{eq: intersec difference and ball}.
	Note that then also $g_{n,i}\notin \bigcup_{l=1}^n T_l$ because $\tau(g_{n,i})\notin \ol{\bigcup_{l=1}^n H_l}$.
	Those $g_{n,i}$ ($i=1,\ldots,m_n$) satisfy \eqref{eq: covering property}.\\
	Now pick $F_{k_{n+1}}$ such that
	\begin{itemize}
		\item $r_{n+1} := \lfloor (1-\gamma) \# F_{k_{n+1}} / 2^{n+1}\rfloor \geq \# S_n \cdot 2^{\# S_n} + m_n + 2$
		\item $G_{n+1} := \bigcup_{j=1}^{2^{\# S_n}} S_{n,j} \cup \{g_{n,i}: i=1,\ldots,m_n\} \cup \{g_n^*\} \tm F_{k_{n+1}}$
	\end{itemize}
	Note that we have $G_{n+1} \tm F_{k_{n+1}}\setminus \bigcup_{l=1}^n T_l$.
	Pick $\varepsilon_{n+1}\in \mathcal{E}$ with $\varepsilon_{n+1}<\varepsilon_n/2$ such that
	\begin{align}\label{eq: properties varepsilon_{n+1}}
	&\text{All balls } B_{\varepsilon_{n+1}}(f) \:(f\in F_{k_{n+1}}) \text{ are pairwise disjoint,} \nonumber\\
	&B_{\varepsilon_{n+1}}(g)\cap \bigcup_{l=1}^n H_l =\emptyset \text{ for all } g\in F_{k_{n+1}}\setminus \bigcup_{l=1}^n T_l \text{ and}\\
	&B_{\varepsilon_{n+1}}(s) \cap B_{\varepsilon_n}(g_n^*) = \emptyset \text{ for all } s\in \bigcup_{j=1}^{2^{\# S_n}} S_{n,j} \nonumber
	\end{align}
	Note that the second property can be guaranteed since $\partial \kl \bigcup_{l=1}^n H_l \kr \cap \tau(G) = \emptyset$ (as $\bigcup_{l=1}^n H_l$ is a finite union of $\varepsilon_l$-balls ($l=1,\ldots,n$)).
	Similarly, the third property can be realized since $\partial B_{\varepsilon_n}(g_n^*) \cap \tau(G) = \emptyset$.\\
	Lemma~\ref{lem: strictly decreasing balls} yields that by making $\varepsilon_{n+1}$ smaller, if necessary, we can also assume that there exists $g_{n+1}^*\in B_{\varepsilon_n}(g_n^*)\cap \tau(G)$ such that
	\begin{equation}\label{eq: properties g_{n+1}^* (1)}
		B_{\varepsilon_{n+1}}(g_{n+1}^*) \tm B_{\varepsilon_n}(g_n^*) \setminus  B_{\varepsilon_{n+1}}(G_{n+1}) = B_{\varepsilon_n}(g_n^*) \setminus \kl B_{\varepsilon_{n+1}}(g_n^*) \cup \bigcup_i B_{\varepsilon_{n+1}}(g_{n,i})\kr
	\end{equation}
	as well as
	\begin{equation}\label{eq: properties g_{n+1}^* (2)}
		B_{\varepsilon_{n+1}}(g_{n+1}^*) \cap B_{\varepsilon_{n+1}}(F_{k_{n+1}}) = \emptyset,
	\end{equation}
	so that in particular $g_{n+1}^*\notin F_{k_{n+1}}$.\\
	Fix a $G_{n+1}$-word $\mathfrak{W}_{n+1}$ which is 0 on all the $g_{n,i}$, which is 1 on $g_n^*$ and such that every $S_n$-word appears on the $S_{n,j}$.
	Let
	\begin{align*}
		U_{n+1} &= \bigcup \{B_{\varepsilon_{n+1}}(f) : f\in G_{n+1} \text{ and } \mathfrak{W}_{n+1}(f)=1\}\\
		V_{n+1} &= \bigcup \{B_{\varepsilon_{n+1}}(f) : f\in G_{n+1} \text{ and } \mathfrak{W}_{n+1}(f)=0\}.
	\end{align*}
	We define $H_{n+1} = U_{n+1}\cup V_{n+1}$, $H_{n+1}^* = H_{n+1}\cup B_{\varepsilon_{n+1}}(g_{n+1}^*)$, $T_{n+1}=\tau^{-1}(H_{n+1})$, $T_{n+1}^* = \tau^{-1}(H_{n+1}^*)$ and $\widehat{T}_{n+1}=\tau^{-1}(B_{\varepsilon_{n+1}}(1_H))$ analogously to the base case.
	Furthermore, let $S_{n+1} = F_{k_{n+1}} \setminus \bigcup_{k=1}^{n+1} T_k$.
	For the inductive construction it remains to prove (1-(n+1)) to (6-(n+1)):\\
	(1-(n+1)) Note that $\# (G_{n+1}\cup \{g_{n+1}^*\}) = r_{n+1}$.
	The claim now follows analogously to the base case, using \eqref{eq: properties varepsilon_{n+1}} (i.e.\ disjointness of the $\varepsilon_{n+1}$-balls).\\
	(2-(n+1)) Observe that
	\begin{align*}
		\# S_{n+1} &= \# F_{k_{n+1}} - \sum_{l=1}^{n+1} \# (F_{k_{n+1}} \cap T_l)
		= \# F_{k_{n+1}} - \sum_{l=1}^{n+1} \sum_{h\in F_{k_{n+1}}\cap \Gamma_{k_l}} \# (F_{k_l}h\cap T_l)\\
		\overset{\text{(1-(l))}}&{\geq} \#F_{k_{n+1}} - \sum_{l=1}^{n+1} \frac{\# F_{k_{n+1}}}{\# F_{k_l}}\cdot r_l
		\geq \# F_{k_{n+1}} - \sum_{l=1}^{n+1} (1-\gamma) \frac{\# F_{k_{n+1}}}{2^l}\\
		&\geq \gamma\cdot \# F_{k_{n+1}}.
	\end{align*}
	(3-(n+1)) From (5-(n)) it follows that
	$$B_{\varepsilon_n}(g_n^*) \setminus \bigcup_{l=1}^{n+1}H_l = B_{\varepsilon_n}(g_n^*) \setminus H_{n+1}.$$
	Moreover, by \eqref{eq: properties varepsilon_{n+1}} we have $B_{\varepsilon_{n+1}}(f) \cap B_{\varepsilon_n}(g_n^*) = \emptyset$ for all $f\in \bigcup_{j}S_{n,j}$.
	Therefore,
	\begin{align*}
	B_{\varepsilon_n}(g_n^*) \setminus H_{n+1} &=B_{\varepsilon_n}(g_n^*) \setminus \kl B_{\varepsilon_{n+1}}(g_n^*) \cup \bigcup_{i=1}^{m_n} B_{\varepsilon_{n+1}}(g_{n,i}) \kr\\
	&=B_{\varepsilon_n}(g_n^*) \setminus \kl B_{\varepsilon_{n+1}}(g_n^*) \cup V_{n+1} \kr.
	\end{align*}
	since we defined $\mathfrak{W}_{n+1}(g_{n,i})=0$ for all $i$.
	(4-(n+1)) This is easy because
	$$ F_{k_n} = S_n \uplus (F_{k_n}\setminus S_n) \tm S_{n,1} \cup \bigcup_{l=1}^n T_l \tm \bigcup_{l=1}^{n+1} T_l.$$
	(5-(n+1)) Note that by (5-(n)) and $B_{\varepsilon_{n+1}}(g_{n+1}^*) \tm B_{\varepsilon_n}(g_n^*)$ we have $B_{\varepsilon_{n+1}}(g_{n+1}^*)\tm G \setminus \bigcup_{l=1}^n H_l$.
	Furthermore, $B_{\varepsilon_{n+1}}(g_{n+1}^*) \cap H_{n+1}\tm B_{\varepsilon_{n+1}}(g_{n+1}^*) \cap B_{\varepsilon_{n+1}}(F_{k_{n+1}}) =\emptyset$ now shows (5-(n+1)).\\
	(6-(n+1))
	Let $s \in S_{n+1}$ and $t \in \widehat{T}_{n+1}= \tau^{-1}(B_{\varepsilon_{n+1}}(1_H))$.
	We see that $\tau(st) = \tau(s) \tau(t) \in B_{\varepsilon_{n+1}}(\tau(s))$.
	Now the first part of \eqref{eq: properties varepsilon_{n+1}} together with \eqref{eq: properties g_{n+1}^* (2)} yield $\tau(st)\in H\setminus H_{n+1}^*$ whereas the second property of \eqref{eq: properties varepsilon_{n+1}} gives us $\tau(st)\in H\setminus \bigcup_{l=1}^n H_l$, showing (6-(n+1)).\\
	This settles the inductive construction.
\end{construction}
We thus obtain sets $U := \bigcup_{n\in\N} U_n$ and $V := \bigcup_{n\in\N} V_n$.
As usual, we let $W_\gamma = \ol{U}$ which defines a proper window.
Moreover, we have for every $n\in\N$ by (4-(n+1)) that $F_{k_n}\tm \bigcup_{l=1}^{n+1} T_l$, so that $\tau(F_{k_n}) \tm \bigcup_{l=1}^{n+1} (U_n\cup V_n) \tm U\cup V$.
The fact that $G = \bigcup_{n\in\N} F_{k_n}$ yields $\tau(G) \tm U\cup V \tm H\setminus \partial W_\gamma$, hence genericity of $W_\gamma$.\\
Before taking care of irredundancy, let us show that indeed $\htop(\ol{O_G(x_{W_\gamma})},G) \geq \gamma\cdot \log 2$.
For this, observe that by construction (more precisely definition of the sets $U_{n+1}$ and $V_{n+1}$) the word $\mathfrak{W}_{n+1}$ occurs in the array $x_{W_{\gamma}}$.
This implies that all $S_n$-words appear in $x_{W_{\gamma}}$, i.e.\ $\# \mathcal{B}_{S_n}(\ol{O_G(x_{W_{\gamma}})}) = 2^{\# S_n}$.
Hence, since $S_n\tm F_{k_n}$
	\[\# \mathcal{B}_{F_{k_n}} := \# \mathcal{B}_{F_{k_n}}(\ol{O_G(x_{W_{\gamma}})}) \geq 2^{\# S_n}.\]
Therefore,
\begin{equation}\label{eq: entropy estimate}
	\htop(\ol{O_G(x_{W_\gamma})},G) = \lim_{n\to\infty} \frac{\log \# \mathcal{B}_{F_{k_n}}}{\# F_{k_n}} \geq \lim_{n\to\infty}  \frac{\log 2^{\# S_n}}{\# F_{k_n}} \overset{(2-(n))}{\geq} \gamma\cdot \log 2.
\end{equation}
Recall that our goal is to find a path $[0,1]\to \mathcal{W}(H),\: t\mapsto W(t)$ such that any possible entropy between $0$ and $\gamma\cdot\log 2$ is attained along this path.
We have already constructed the endpoint of this path, namely $W_\gamma$.
It thus remains to find an appropriate starting point $W_0$.
It turns out that letting $W_0 = \ol{\bigcup_{n\in\N} B_{\varepsilon_n}(g_{n-1}^*)}$ does the job.
\begin{rem}
	Note that in order to prove Theorem~\ref{thm: realize entropy intro}, it does not suffice to pick an arbitrary window $W_0$ that admits zero entropy, because Theorem~\ref{thm: generic windows path connected} does not guarantee that the elements on the path remain proper and irredundant.
	One can compare this with \cite{LackaStraszak2018QuasiUniform}, where the authors do exactly that, by (implicitly) choosing the starting point as some arbitrary Toeplitz array which admits zero entropy.
	This is why their result (just like the original result by Krieger) only gives Toeplitz arrays \textbf{relative} to a specific sequence $\mathbf{\Gamma}$ and not with $\mathbf{\Gamma}$ as its period structure.
\end{rem}
\begin{lem}
	$W_0$ is proper, generic and regular.
	In particular, $\htop(\ol{O_G(x_{W_0})},G) = 0$.
\end{lem}
\begin{proof}
	First, note that by $B_{\varepsilon_{n+1}}(g_{n+1}^*) \tm B_{\varepsilon_n}(g_n^*)$ (c.f. \eqref{eq: properties g_{n+1}^* (1)}) and $\varepsilon_n \xrightarrow{n\to\infty} 0$, the sequence $(g_n^*)_{n\in\N}$ is a Cauchy sequence in $H$, hence has a limit $\xi^*$.
	We claim the following:
	\begin{equation}\label{eq: boundary of W_0}
		\partial W_0 \tm \bigcup_{n\in \N}\partial B_{\varepsilon_n}(g_{n-1}^*) \cup \{\xi^*\}.
	\end{equation}
	To prove the claim let $\eta \in \partial W_0$ but assume $\eta\notin \bigcup_{n\in\N} \partial B_{\varepsilon_n}(g_{n-1}^*)$.
	If $\eta \in B_{\varepsilon_n}(g_{n-1}^*)$ for some $n$, then $\eta \in \mathrm{int}(W_0)$, a contradiction.
	This shows that $\delta_n := \inf\{ d(\eta,\zeta): \zeta\in B_{\varepsilon_n}(g_{n-1}^*)\} > 0$ for all $n\in\N$.
	Since $\eta \in W_0 = \ol{\bigcup_{n\in\N} B_{\varepsilon_n}(g_{n-1}^*)}$, we clearly see that
	\[\liminf_{n\to\infty} \delta_n = \inf\{\delta_n: n\in\N\} = 0.\]
	Therefore, $\liminf_{n\to\infty} d(\eta,g_{n-1}^*) \leq \liminf_{n\to\infty} (\varepsilon_n+\delta_n) = 0$, so that $\eta = \xi^*$ as desired.
	Now, \eqref{eq: boundary of W_0} immediately implies regularity since $\varepsilon_n\in \mathcal{E}$ for all $n\in\N$.
	For genericity, we observe that it suffices to show that $\xi^*\notin \tau(G)$.
	Clearly, $\xi^*\in \bigcap_{n\in\N} B_{\varepsilon_n}(g_n^*)$.
	If we assume that $\xi^* = \tau(g^*)$ for some $g^* \in G$, then the observations right after the construction imply $\xi^* \in U_n \cup V_n = H_n$ for some $n\in\N$.
	This however contradicts (5-(n)), showing genericity.
	Since properness is clear, this concludes the proof.
\end{proof}

Furthermore, observe that since we chose the word $\mathfrak{W}_n$ to be 1 on $g_{n-1}^*$, we have $B_{\varepsilon_n}(g_{n-1}^*)\tm U_n$ for every $n\in\N$.
Therefore, $W_0 \tm W_\gamma$.
Now let $t\mapsto W(t)$ be the path from $W_0$ to $W_\gamma$ defined as in the proof of Theorem~\ref{thm: generic windows path connected}, that is
$$ W(t) = (W_\gamma \cap A(t)) \cup (W_0\setminus \mathrm{int}(A(t))$$
where $A(t)$ is defined in the proof of Proposition~\ref{prop: path between empty and full}.
By Theorem~\ref{thm: generic windows path connected}, we have $W(t) \in \mathfrak{B}_{\mathrm{gen}}(H)\cap \mathfrak{B}_{\mathrm{closed}}(H)$.
Moreover, clearly $W_0\tm W(t)\tm W_\gamma$ for all $t\in [0,1]$.
We now show that we can produce the same array $x_{W(t)}$ with a proper, generic window $W^*(t)$ which still satisfies $W_0\tm W^*(t)\tm W_\gamma$.

\begin{lem}\label{lem: properfication}
	For any $t\in[0,1]$ there exists $W^*(t)\tm H$ which is proper, generic and which satisfies $W_0\tm W^*(t)\tm W_\gamma$ as well as $x_{W^*(t)} = x_{W(t)}$.
\end{lem}
\begin{proof}
	We define $W^*(t) = \ol{\mathrm{int}(W(t))}$.
	Since $W(t)$ is closed, we have $W^*(t)\tm W(t) \tm W_\gamma$.
	This also implies $\mathrm{int}(W^*(t))\tm \mathrm{int}(W(t)))$.
	Conversely, clearly $\mathrm{int}(W(t))\tm \mathrm{int}(W^*(t))$ by definition, so that $\mathrm{int}(W^*(t))= \mathrm{int}(W(t))$.
	In particular, $W^*(t)$ is proper.
	Moreover
	$$\partial W^*(t) = W^*(t) \setminus \mathrm{int}(W^*(t)) \tm W(t)\setminus \mathrm{int}(W(t)) = \partial W(t),$$
	so that $W^*(t)$ is generic.
	Note that $W_0 \tm W(t)$ which implies $W_0 = \ol{\mathrm{int}(W_0)}\tm W^*(t)$, so that in particular $W_0\tm W^*(t) \tm W_\gamma$.
	The fact that $x_{W^*(t)} = x_{W(t)}$ follows readily from genericty of both windows $W^*(t)$ and $W(t)$ together with the fact that $\mathrm{int}(W^*(t)) =\mathrm{int}(W(t))$.
\end{proof}

\begin{lem}\label{lem: window on path irredundant}
	Let $W\tm H$ be a window with $W_0\tm W\tm W_\gamma$.
	Then, $W$ is irredundant.
\end{lem}
\begin{proof}
	Let $\xi\in H$ such that $W\xi = W$.
	We show by induction that $d_H(\xi,1_H) < \varepsilon_n$ for all $n\in\N_0$.\\
	The base case ($n=0$) is clear since $B_{\varepsilon_0}(1_H) = B_{\varepsilon_0}(g_0^*) = H$.\\
	\underline{$n\to n+1$: }
	Observe that $B_{\varepsilon_{n+1}}(g_n^*) \tm W_0\tm W$.
	Hence, the assumption that $W\xi=W$ implies
	\begin{equation}\label{eq: shifted ball not in V}
	B_{\varepsilon_{n+1}}(g_n^*\xi) \tm W \tm W_\gamma \tm H\setminus V\tm H\setminus V_{n+1}.
	\end{equation}
	Suppose for a contradiction that $d_H(\xi,1_H)\geq \varepsilon_{n+1}$.
	By the induction hypothesis we thus have $\varepsilon_{n+1}\leq d_H(\xi,1_H) < \varepsilon_n$, so that $\tau(g_n^*)\xi \in B_{\varepsilon_n}(g_n^*)\setminus B_{\varepsilon_{n+1}}(g_n^*)$
	Together with \eqref{eq: shifted ball not in V} we even have
	$$ \tau(g_n^*) \xi \in B_{\varepsilon_n}(g_n^*) \setminus (B_{\varepsilon_{n+1}}(g_n^*)\cup V_{n+1}) \overset{(3-(n+1))}{=}B_{\varepsilon_n}(g_n^*)\setminus \bigcup_{l=1}^{n+1} H_l.$$
	By choice of the $g_{n+1,i}$ (see \eqref{eq: covering property}) we have $\tau(g_n^*)\xi \in B_{\varepsilon_{n+1}}(g_{n+1,i})$ for some $i$, i.e.\ $\tau(g_{n+1,i}) \in B_{\varepsilon_{n+1}}(\tau(g_n^*)\xi)$.
	However, recall that in the inductive construction of $W_\gamma$ we chose the word $\mathfrak{W}_{n+2}$ to be 0 on all the $g_{n+1,i}$.
	Hence, $\tau(g_{n+1,i})\in V_{n+2}\tm V$, which contradicts $B_{\varepsilon_{n+1}}(\tau(g_n^*)\xi)\tm H\setminus V$.
	As $\varepsilon_n \to 0$, we have shown that $\xi = 1_H$ as desired.
\end{proof}

We can now collect all results of this section and wrap up the proof of Theorem~\ref{thm: realize entropy intro}.
\begin{proof}[Proof of Theorem~\ref{thm: realize entropy intro}]
	Let $h \in [0,\log 2)$ be arbitrary and let $(H,\tau)$ be an arbitrary metric group compactification of $G$.
	We let $\gamma \in [0,1)$ be such that $h = \gamma \cdot \log 2$.
	By \eqref{eq: entropy estimate} we have $\htop(\ol{O_G(x_{W_\gamma})},G)\geq h$.
	Let $t\mapsto W(t)$ denote the path between $W_0$ and $W_\gamma$ from above.
	Corollary~\ref{cor: entropy continuous} and Theorem~\ref{thm: generic windows path connected} imply that the mapping $[0,1] \to [0,\infty), \: t \mapsto \htop(\ol{O_G(x_{W(t)})},G)$ is continuous.
	Hence, by the intermediate value theorem there exists $t\in [0,1]$ such that $\htop(\ol{O_G(x_{W(t)})},G) = h$.
	Due to Lemma~\ref{lem: properfication} and Lemma~\ref{lem: window on path irredundant}, there exists a proper, generic and irredundant window $W^*(t)$ such that $x_{W^*(t)} = x_{W(t)}$.
	Therefore, $x_{W(t)}$ is almost automorphic over $H$ by Theorem~\ref{thm: 1-1 windows almost automorphic}.
	This concludes the proof.
\end{proof}

\bibliographystyle{alpha}

\begin{thebibliography}{DDSMS99}

\bibitem[BFK97]{BlFoKu1997Cellular}
F.~Blanchard, E.~Formenti, and P.~K{\r u}rka.
\newblock Cellular automata in the {C}antor, {B}esicovitch, and {W}eyl
  topological spaces.
\newblock {\em Complex Systems}, 11(2):107--123, 1997.

\bibitem[BHP18]{BjHaPo2018AperiodicOrder}
M.~Bj\"orklund, T.~Hartnick, and F.~Pogorzelski.
\newblock Aperiodic order and spherical diffraction, {I}: auto-correlation of
  regular model sets.
\newblock {\em Proc.\ Lond.\ Math.\ Soc.}, 116(4):957--996, 2018.

\bibitem[BJL16]{BaakeJagerLenz2016ToeplitzAndModel}
M.~Baake, T.~J\"ager, and D.~Lenz.
\newblock {T}oeplitz flows and model sets.
\newblock {\em Bull. Lond. Math. Soc.}, 48:691--698, 2016.

\bibitem[BLM07]{BaLeMo2007Model}
M.~Baake, D.~Lenz, and R.~Moody.
\newblock Characterization of model sets by dynamical systems.
\newblock {\em Ergodic Theory and Dynamical Systems}, 27(2):341--382, 2007.

\bibitem[CBCG24]{CecchiCortezGomez2024Invariant}
P.~Cecchi~Bernales, M.~I. Cortez, and J.~Gómez.
\newblock Invariant measures of {T}oeplitz subshifts on non-amenable groups.
\newblock {\em Ergodic Theory and Dynamical Systems}, 44:3186--3215, 2024.

\bibitem[CDGJ25]{CortezDrewloGomezJager2025Model}
M.~I. Cortez, J.~Drewlo, J.~Gómez, and T.~J{\"a}ger.
\newblock Cut and project schemes and {T}oeplitz subshifts.
\newblock {\em In preparation}, 2025.

\bibitem[CFMM97]{CaFoMaMa1997Besicovitch}
G.~Cattaneo, E.~Formenti, L.~Margara, and J.~Mazoyer.
\newblock A shift-invariant metric on {${S}^{\mathbb{Z}}$} inducing a
  nontrivial topology.
\newblock In {\em MFCS 1997}, volume 1295 of {\em LNCS}, pages 179--188, 1997.

\bibitem[CP08]{CortezPetite2008Odometers}
M.~I. Cortez and P.~Petite.
\newblock {$G$}-odometers and their almost one-to-one extensions.
\newblock {\em J. Lond. Math. Society}, 78:1--20, 2008.

\bibitem[CP14]{CortezPetite2014Invariant}
M.~I. Cortez and P.~Petite.
\newblock Invariant measures and orbit equivalence for generalized {T}oeplitz
  subshifts.
\newblock {\em Groups, Geom., Dyn.}, 8:1007--1045, 2014.

\bibitem[DDSMS99]{DixonDusautoyMannSegal2003Analytic}
J.~D. Dixon, M.~P.~F. Du~Sautoy, A.~Mann, and D.~Segal.
\newblock {\em Analytic pro-$p$ groups (2nd ed.)}.
\newblock Cambridge University Press, 1999.

\bibitem[DI88]{DoIw1988QuasiUniform}
T.~Downarowicz and A.~Iwanik.
\newblock Quasi-uniform convergence in compact dynamical systems.
\newblock {\em Studia Mathematica}, 89:11--25, 1988.

\bibitem[DJL26]{DrJaLe2025Model}
J.~Drewlo, T.~J\"ager, and D.~Lenz.
\newblock A survey on model sets.
\newblock {\em In preparation}, 2026.

\bibitem[Dow11]{Do2011Entropy}
T.~Downarowicz.
\newblock {\em Entropy in Dynamical Systems}.
\newblock Cambridge University Press, 2011.

\bibitem[EG60]{ElGo1960Homomorphisms}
R.~Ellis and W.~H. Gottschalk.
\newblock Homomorphisms of transformation groups.
\newblock {\em Trans.\ Amer.\ Math.\ Soc.}, 94:258--271, 1960.

\bibitem[Fal90]{Falconer1990Fractal}
K.~J. Falconer.
\newblock {\em Fractal geometry : {M}athematical foundations and applications}.
\newblock Wiley, 1990.

\bibitem[FG20]{FuGr2020Constant}
G.~Fuhrmann and M.~Gr{\"o}ger.
\newblock Constant length substitutions, iterated function systems and amorphic
  complexity.
\newblock {\em Math.\ Z.}, 295(4):1385--1404, 2020.

\bibitem[FGJ16]{FuGrJa2016Amorphic}
G.~Fuhrmann, M.~Gr{\"o}ger, and T.~J{\"a}ger.
\newblock Amorphic complexity.
\newblock {\em Nonlinearity}, 29(2):528--565, 2016.

\bibitem[FGJK23]{FuGrJaKw2023Amorphic}
G.~Fuhrmann, M.~Gr{\"o}ger, T.~J{\"a}ger, and D.~Kwietniak.
\newblock Amorphic complexity of group actions with applications to
  quasicrystals.
\newblock {\em Trans.\ Amer.\ Math.\ Soc.}, 376(4):2395--2418, 2023.

\bibitem[FGJO21]{FuGlJaOe2021IrregularModel}
G.~Fuhrmann, E.~Glasner, T.~J\"ager, and C.~Oertel.
\newblock Irregular model sets and tame dynamics.
\newblock {\em Trans.\ Amer.\ Math.\ Soc.}, 374(5):3703--3734, 2021.

\bibitem[HM20]{HoMo2020CompactGroups}
K.~H. Hofmann and S.~A. Morris.
\newblock {\em The structure of compact groups -- a primer for the student -- a
  handbook for the expert}.
\newblock De Gruyter, 2020.

\bibitem[JLO19]{JaLeOe2019ModelPositive}
T.~J\"ager, D.~Lenz, and C.~Oertel.
\newblock Model sets with positive entropy in euclidean cut and project
  schemes.
\newblock {\em Ann.\ Sci.\ ENS}, 52(5):1073--1106, 2019.

\bibitem[KK25]{KasjanKeller2025Besicovitch}
S.~Kasjan and G.~Keller.
\newblock Besicovitch covering numbers for $\mathcal{B}$-free and other shifts.
\newblock {\em arXiv preprint
  \href{https://doi.org/10.48550/arXiv.2505.09253}{arXiv:2012.01396}}, 2025.

\bibitem[KR19]{KeRi2019WeakModel}
G.~Keller and C.~Richard.
\newblock Periods and factors of weak model sets.
\newblock {\em Israel J.\ Math.}, 229:85--132, 2019.

\bibitem[Kri07]{Krieger2007Toeplitz}
F.~Krieger.
\newblock Sous-décalages de {T}oeplitz sur les groupes moyennables
  résiduellement finis.
\newblock {\em J. Lond. Math. Society}, 75:447--462, 2007.

\bibitem[Kul24]{Ku2024Amorphic}
M.~Kulczycki.
\newblock Amorphic complexity can take any nonnegative value.
\newblock {\em J.\ Dynam.\ Differential Equations}, 36:1115--1121, 2024.

\bibitem[LP03]{LaPl2003Repetitive}
J.~C. Lagarias and P.~A.~B. Pleasants.
\newblock Repetitive {D}elone sets and quasicrystals.
\newblock {\em Ergodic Theory and Dynamical Systems}, 23:831--867, 2003.

\bibitem[LS03]{LubotzkySegal2003SubgroupGrowth}
A.~Lubotzky and D.~Segal.
\newblock {\em Subgroup Growth}.
\newblock Birkh{\" a}user Basel, 2003.

\bibitem[LS18]{LackaStraszak2018QuasiUniform}
M.~\L{}ącka and M.~Straszak.
\newblock Quasi-uniform convergence in dynamical systems generated by an
  amenable group action.
\newblock {\em J.\ Lond.\ Math.\ Soc.}, 98:687--707, 2018.

\bibitem[Mey72]{Me1972Algebraic}
Y.~Meyer.
\newblock {\em Algebraic Numbers and Harmonic Analysis}.
\newblock North-Holland Publishing Co., 1972.

\bibitem[Pie84]{Pier1984Amenable}
J.-P. Pier.
\newblock {\em Amenable Locally Compact Groups}.
\newblock Wiley, 1984.

\bibitem[Sch00]{Sc2000GeneralizedModel}
M.~Schlottmann.
\newblock Generalized model sets and dynamical systems.
\newblock In {\em Directions in Mathematical Quasicrystals}, volume~13 of {\em
  CRM Monograph Series}, pages 143--159. American Mathematical Society, 2000.

\end{thebibliography}

\end{document}